\documentclass[smallextended,draft]{svjour3} 
\usepackage{amsmath, amssymb, latexsym, amscd,amsfonts,amstext}
\usepackage[mathscr]{eucal}
\usepackage{color,calc,graphicx,soul}

\def\O{{\mathscr O}}

\journalname{Vietnam Journal of Mathematics}
\allowdisplaybreaks
\begin{document}
\numberwithin{equation}{section}
\numberwithin{lemma}{section}
\numberwithin{theorem}{section}
\numberwithin{corollary}{section}
\numberwithin{remark}{section}
\setlength{\baselineskip}{15pt}
\title{\Large\bf{Parallel hybrid iterative methods for  }\\
\bf {variational inequalities, equilibrium problems }\\
\bf{and common fixed point problems }}
\titlerunning{Parallel hybrid iterative methods for VIs, EPs, and FPPs}
\author {P. K. Anh \and D.V. Hieu} 
\dedication{\it Dedicated to Professor Nguyen Khoa Son's 65th Birthday}
%\authorrunning{Short form of author list} 
\institute{ P. K. Anh (Corresponding author)\and D.V. Hieu \at College of Science, Vietnam National University, Hanoi,  $334$ 
Nguyen Trai, Thanh Xuan, Hanoi, Vietnam \\\email{anhpk@vnu.edu.vn, dv.hieu83@gmail.com}} 
\date{}
\maketitle {}
\begin{abstract}
\noindent In this paper we propose two strongly convergent parallel hybrid iterative  methods for finding a common element of the set of fixed 
points of a family of quasi $\phi$-asymptotically nonexpansive mappings $\{F(S_j)\}_{j=1}^N$, the set of solutions of variational 
inequalities $\{VI(A_i,C)\}_{i=1}^M$ and the set of solutions of equilibrium problems $\{EP(f_k)\}_{k=1}^K$ in uniformly smooth and 2-uniformly 
convex Banach spaces. A numerical experiment is given to verify the efficiency of the proposed parallel algorithms.
\end{abstract}
\section{Introduction}
\setcounter{lemma}{0}
\setcounter{theorem}{0}
\setcounter{equation}{0}
Let $C$ be a nonempty closed convex subset of a Banach space $E$. 
The variational inequality for a possibly nonlinear mapping  $A:C\to E$, consists of finding $p^*\in C$ such as
\begin{equation}\label{eq:VIP}
\left\langle Ap^*,p-p^* \right\rangle \ge 0,\quad \forall p\in C.
\end{equation}
The set of solutions of (\ref{eq:VIP}) is denoted by $VI(A,C)$. \\
Takahashi and Toyoda \cite{TT2003} proposed a weakly convergent method for finding a common element of the set of 
fixed points of a nonexpansive mapping and the set of solutions of the variational inequality for an $\alpha$ - inverse strongly monotone 
mapping in a Hilbert space.
\begin{theorem}\label{theo.TakahashiToyoda2003}\cite{TT2003}
Let $K$ be a closed convex subset of a real Hilbert space $H$. Let $\alpha>0$. Let $A$ be an $\alpha$ - inverse strongly-monotone 
mapping of $K$ into $H$, and let $S$ be a nonexpansive mapping of $K$ into itself such that $F(S)\bigcap VI(K,A)\ne \emptyset$. Let 
$\left\{x_n\right\}$ be a sequence generated by
$$
\left\{
\begin{array}{ll}
&x_0\in K,\\
&x_{n+1}=\alpha_n x_n+(1-\alpha_n)SP_K(x_n-\lambda_n Ax_n),
\end{array}
\right.
$$
for every $n=0,1,2,\ldots$,  where $\lambda_n \in [a,b]$ for some $a,b\in(0,2\alpha)$ and $\alpha_n \in [c,d]$ for some $c,d\in(0,1)$. Then,
 $\left\{x_n\right\}$ converges weakly to $z\in F(S)\bigcap VI(K,A)$, where $z=\lim_{n\to \infty}P_{F(S)\bigcap VI(K,A)}x_n$.
\end{theorem}
In 2008, Iiduka and Takahashi \cite{IT2008} considered  problem $(\ref{eq:VIP})$ in a 2-uniformly convex, uniformly smooth Banach 
space under the following assumptions:
\begin{itemize}
\item [(V1)] $A$ is $\alpha$-inverse-strongly-monotone.
\item [(V2)]$VI(A,C)\ne \emptyset$.
\item [(V3)] $||Ay||\le||Ay-Au||$ for all $y\in C$ and $u\in VI(A,C)$.
\end{itemize}
\begin{theorem}\label{thm:IT2008}\cite{IT2008}
Let $E$ be a  $2$-uniformly convex, uniformly smooth Banach space whose duality mapping $J$ is weakly sequentially continuous, and 
let $C$  be a 
nonempty, closed convex subset of  $E$ . Assume that $A$ is a mapping of $C$ into $E^*$ satisfing conditions $(V1)-(V3)$. 
Suppose that $x_1=x\in C$ and $\left\{x_n\right\}$ is given by
$$
x_{n+1}=\Pi_C J^{-1}(Jx_n-\lambda_n Ax_n)
$$
for every $n = 1 , 2 ,... , $ where $\left\{\lambda_n\right\}$ is a sequence of positive numbers. If $\lambda_n$ is chosen so that 
$\lambda_n\in [a,b]$ for some
$a,b$ with $0<a<b<\frac{c^2\alpha}{2}$ , then the sequence $\left\{x_n\right\}$ converges weakly to some element $z$ in $VI(C, A).$  
Here $1/c$ is
the $2$-uniform convexity constant of $E$, and $z=\lim_{n\to\infty}\Pi_{VI(A,C)}x_n$.
\end{theorem}

In 2009,  Zegeye and Shahzad \cite{ZS2009} studied the following hybrid 
iterative algorithm in a 2-uniformly convex and uniformly smooth Banach space for finding a common element of the set of fixed points of a 
weakly relatively nonexpansive mapping $T$ and the set of solutions of a 
variational inequality involving an $\alpha$-inverse strongly monotone mapping $A$:
$$
\left \{
\begin{array}{ll}
&y_n=\Pi_C\left(J^{-1}(Jx_n-\lambda_n Ax_n)\right),\\
&z_n=Ty_n,\\
&H_0=\left\{v\in C:\phi(v,z_0)\le\phi(v,y_0)\le\phi(v,x_0)\right\},\\
&H_n=\left\{v\in H_{n-1}\bigcap W_{n-1}:\phi(v,z_n)\le\phi(v,y_n)\le\phi(v,x_n)\right\},\\
&W_0=C,\\
&W_n=\left\{v\in H_{n-1}\bigcap W_{n-1} : \left\langle x_n-v,Jx_0-Jx_n\right\rangle \ge 0\right\},\\
&x_{n+1}=P_{H_n\bigcap W_n}x_0,n\ge 1,
\end{array}
\right.
$$
where $J$ is the normalized duality mapping on $E$. The strong convergence of  $\left\{x_n\right\}$ to $\Pi_{F(T)\bigcap 
VI(A,C)}x_0$ has been established.\\
Kang, Su, and Zhang \cite{KSZ2010} extended this algorithm to a weakly relatively nonexpansive mapping, a 
variational inequality  and an equilibrium problem. Recently, Saewan and Kumam \cite{SK2012} have constructed a sequential hybrid block 
iterative algorithm for an infinite family of closed and uniformly quasi $\phi$- asymptotically nonexpansive mappings, a variational 
inequality for an $\alpha$ -inverse-strongly monotone mapping, and a system of equilibrium problems.

Qin, Kang, and Cho \cite{QKC2009} considered the following sequential hybrid method for a pair of inverse strongly monotone and a 
quasi $\phi$-nonexpansive mappings in a 2-uniformly convex and uniformly smooth Banach space:
$$
\left \{
\begin{array}{ll}
&x_0=E, C_1=C, x_1=\Pi_{C_1}x_0,\\
&u_n=\Pi_C\left(J^{-1}(Jx_n-\eta_n Bx_n)\right),\\
&z_n=\Pi_C\left(J^{-1}(Ju_n-\lambda_n Au_n)\right),\\
&y_n=Tz_n,\\
&C_{n+1}=\left\{v\in C_n:\phi(v,y_n)\le\phi(v,z_n)\le \phi(v,u_n)\le\phi(v,x_n)\right\},\\
&x_{n+1}=\Pi_{C_{n+1}}x_0,n\ge 0.
\end{array}
\right.
$$
They proved that the sequence $\left\{x_n\right\}$ converges strongly to $\Pi_Fx_0$, where $F={F(T)\bigcap VI(A,C)\bigcap VI(B,C)}.$

Let $f$ be a bifunction from $C\times C$ to a set of real numbers $\mathbb{R}$. The equilibrium problem for $f$ consists of  finding an
element  $\widehat{x}\in C$, such that
\begin{equation}\label{eq:EP}
f(\widehat{x},y)\ge 0,\, \forall y\in C.
\end{equation}
The set of solutions of the equilibrium problem $(\ref{eq:EP})$ is denoted by $EP(f)$. Equilibrium problems include 
several problems such as: variational inequalities, optimization problems, fixed point problems, ect. In recent years, equilibrium problems have been studied widely and several solution methods have been proposed (see \cite{ABH2014,KSZ2010,SK2012,SLZ2011,TT2007}).
On the other hand, for finding a common element in $F(T)\bigcap EP(f)$, Takahashi and Zembayashi \cite{TZ2009} introduced the 
following algorithm in a uniformly smooth and uniformly convex Banach space:
$$
\left \{
\begin{array}{ll}
&x_0\in C,\\
&y_n=J^{-1}(\alpha_n Jx_n+(1-\alpha_n )JTy_n),\\
&u_n \in C, {\rm such~ that}, f(u_n,y)+\frac{1}{r_n}\left\langle y-u_n,Ju_n-Jy_n\right\rangle \ge 0 \quad \forall y\in C,\\
&H_n=\left\{v\in C:\phi(v,u_n)\le\phi(v,x_n)\right\},\\
&W_n=\left\{v\in C: \left\langle x_n-v,Jx_0-Jx_n\right\rangle \ge 0\right\},\\
&x_{n+1}=P_{H_n\bigcap W_n}x_0,n\ge 1.
\end{array}
\right.
$$
The strong convergence of  the sequences $\left\{x_n\right\}$ and $\left\{u_n\right\}$ to $\Pi_{F(T)\bigcap EP(f)}x_0$ has been established.\\
Recently, the above mentioned algorithms have been generalized and modified for finding a common point of the set of solutions of variational 
inequalities, the set of fixed points of quasi $\phi$- (asymptotically) nonexpansive mappings, and the set of solutions of equilibrium 
problems by several authors, such as Takahashi and Zembayashi \cite{TZ2009},
%Cholamjiak \cite{C2009}, Zegeye \cite{Z2010}, Wang and Zhou \cite{WZ2011},
Wang et al. \cite{WKC2012} and others.% \cite{IT2004,Y2001}.%\cite{Z2012,K1976,KB2013,BD2011,AN2002}.\\
\\Very recently, Anh and Chung \cite{AC2013} have considered the following parallel hybrid method for a finite family of 
relatively nonexpansive mappings $\{T_i\}_{i=1}^N$:
$$
\left \{
\begin{array}{ll}
&x_0\in C,\\ 
& y_n^i=J^{-1}(\alpha_n Jx_n+(1-\alpha_n)JT_{i}x_n), \quad i=1,\ldots,N,\\
&i_n=\arg\max_{1\le i\le N}\left\{\left\|y_n^i-x_n \right\|\right\}, \quad  \bar{y}_n := y_n^{i_n},\\
&C_n=\left\{v\in C:\phi(v,\bar{y}_n)\le \phi(v,x_n)\right\},\\
&Q_n=\left\{v\in C : \left\langle Jx_0-Jx_n,x_n-v\right\rangle \ge 0\right\},\\
&x_{n+1}=\Pi_{C_n\bigcap Q_n}x_0,n\ge 0.
\end{array}
\right.
$$
This algorithm was extended, modified and generelized by Anh and Hieu \cite{AH2014} for a finite family of asymptotically quasi 
$\phi$-nonexpansive mappings in Banach spaces. Note that the proposed parallel hybrid methods in \cite{AC2013,AH2014} can be used for solving simultaneuous systems of maximal monotone mappings. Other parallel methods for solving accretive operator equations can be found in \cite{ABH2014}.\\
In this paper, motivated and inspired by the above mentioned results, we propose  two novel parallel iterative methods for finding a common 
element of the set of fixed points of a family of asymptotically quasi $\phi$-nonexpansive mappings $\{F(S_j)\}_{j=1}^N$, the set of
solutions of  variational inequalities $\{VI(A_i,C)\}_{i=1}^M$, and the set of solutions of equilibrium problems $\{EP(f_k)\}_{k=1}^K$ in 
uniformly smooth and 2-uniformly convex Banach spaces, namely:\\
{\bf Method A}\\
\begin{equation}\label{eq:VIP-EP-FPP}
\left \{
\begin{array}{ll}
&x_0\in C\quad\mbox{chosen arbitrarily,}\\
&y_n^i=\Pi_C\left(J^{-1}(Jx_n-\lambda_n A_i x_n)\right), i=1,2,\ldots M, \\
&i_n=\arg\max\left\{||y_n^i-x_n||:i=1,\ldots, M \right\}, \bar{y}_n=y_n^{i_n},\\
&z_n^j=J^{-1}\left(\alpha_n Jx_n+(1-\alpha_n)JS_j^n\bar{y}_n\right),j=1,\ldots, N,\\
&j_n=\arg\max\left\{||z_n^j-x_n||:j=1,\ldots, N\right\}, \bar{z}_n=z_n^{j_n},\\
&u_n^k=T_{r_n}^k \bar{z}_n,k=1,\ldots, K, \\
&k_n=\arg\max\left\{||u_n^k-x_n||:k=1,2,\ldots K\right\}, \bar{u}_n=u_n^{k_n},\\
&C_{n+1}=\left\{z\in C_n:\phi(z,\bar{u}_n)\le \phi(z,\bar{z}_n)\le\phi(z,x_n)+\epsilon_n\right\},\\
&x_{n+1}=\Pi_{C_{n+1}}x_0,n\ge 0,
\end{array}
\right.
\end{equation}
where, $T_rx := z $ is a unique solution to a regularized equlibrium problem \\
$f(z,y)+\frac{1}{r}\langle y-z,Jz-Jx\rangle\geq0, \quad \forall y\in
C.$\\
Further, the control parameter sequences $\left\{\lambda_n\right\},\left\{\alpha_n\right\},\left\{r_n\right\}$ satisfy the conditions
\begin{equation}\label{dk:VIP-EP-FPP}
0\le\alpha_n\le 1, \lim\sup_{n\to\infty}\alpha_n <1,\quad \lambda_n\in [a,b],\quad r_n\ge d,
\end{equation}
for some $a,b\in (0,\alpha c^2 /2), d> 0$ with $1/c$ being the 2-uniform convexity constant of $E.$  Concerning the sequence 
$\{\epsilon_n\}$, we 
consider two cases. If the mappings $\{S_i\}$ are quasi $\phi$-asymptotically nonexpansive, we assume that the solution set $F$ is 
bounded, i.e., there exists a positive number $\omega$, such that $F \subset \Omega := \{u \in C: ||u|| \leq \omega \}$ and put
$\epsilon_n:= (k_n-1)(\omega + ||x_n||)^2$. If the mappings $\{S_i\}$ are quasi $\phi$-nonexapansive, then $k_n = 1$, and we put 
$\epsilon_n = 0$.\\
{\bf Method B}
\begin{equation}\label{eq:VIP-EP-FPP1}
\left \{
\begin{array}{ll}
&x_0\in C\quad\mbox{chosen arbitrarily,}\\
&y_n^i=\Pi_C\left(J^{-1}(Jx_n-\lambda_n A_i x_n)\right), i=1,\ldots, M, \\
&i_n=\arg\max\left\{||y_n^i-x_n||:i=1,\ldots, M \right\}, \bar{y}_n=y_n^{i_n},\\
&z_n=J^{-1}\left(\alpha_{n,0} Jx_n+\sum_{j=1}^N \alpha_{n,j} JS_j^n\bar{y}_n\right),\\
&u_n^k=T_{r_n}^k z_n,  k=1,\ldots, K, \\
&k_n=\arg\max\left\{||u_n^k-x_n||:k=1,\ldots, K\right\}, \bar{u}_n=u_n^{i_n},\\
&C_{n+1}=\left\{z\in C_n:\phi(z,\bar{u}_n)\le\phi(z,x_n)+\epsilon_n\right\},\\
&x_{n+1}=\Pi_{C_{n+1}}x_0,n\ge 0,
\end{array}
\right.
\end{equation}
where, the control parameter sequences $\left\{\lambda_n\right\},\left\{\alpha_{n,j}\right\},\left\{r_n\right\}$ satisfy the conditions
\begin{equation}\label{dk:VIP-EP-FPP1}
0\le\alpha_{n,j}\le 1,\,\sum_{j=0}^N \alpha_{n,j}=1, \quad\lim_{n\to\infty}\inf\alpha_{n,0}\alpha_{n,j}>0,\quad \lambda_n\in [a,b],
\,r_n\ge d.
\end{equation}
In Method A $(\ref{eq:VIP-EP-FPP})$, knowing $x_n$ we find  the intermediate approximations $y_n^i,  i=1,\ldots, M $ in parallel. Using 
the farthest element among $y_n^i$  from $x_n$, we compute $z_n^j,  j=1,\ldots, N$ in parallel. Further, among $z_n^j$, 
we choose the farthest element from $x_n$ and determine solutions of regularized equilibrium problems $u_n^k,  k=1,\ldots,K$ in parallel.
Then the farthest from $x_n$ element among $u_n^k,$ denoted by $\bar{u}_n$ is chosen. 
Based on $\bar{u}_n$, a closed convex subset $C_{n+1}$ is constructed. Finally, the next 
approximation $x_{n+1}$ is defined as the generalized projection of $x_0$ onto $C_{n+1}$. \\
A similar idea of parallelism is employed in Method B $(\ref{eq:VIP-EP-FPP1})$. However, the subset $C_{n+1}$
 in Method B is simpler than that in Method A.\\
The results obtained in this paper extend and modify the corresponding results of Zegeye and Shahzad \cite{ZS2009}, Takahashi and Z
embayashi 
\cite{TZ2009}, Anh and Chung \cite{AC2013}, Anh and Hieu \cite{AH2014} and others.\\
The paper is organized as follows: In Section 2, we collect some definitions and results needed for further investigtion. Section 3 deals
with  the convergence analysis of the methods $(\ref{eq:VIP-EP-FPP})$ and $(\ref{eq:VIP-EP-FPP1})$. In the last section, a novel parallel 
hybrid iterative method for variational inequalities and closed, quasi $\phi$- nonexpansive mappings is studied.
%%%%%%%%%%%%%%%%%%%%%%%%%%%%%%%%%%%%%%%%%%%%%%%%%%%%%%%
\section{Preliminaries}
\setcounter{lemma}{0}
\setcounter{theorem}{0}
\setcounter{equation}{0}
\setcounter{remark}{0}
\setcounter{corollary}{0}
In this section we recall some definitions and results which will be used later. The reader is refered to \cite{AR2006} for more details. 
\begin{definition}
A Banach space $E$ is called
\begin{itemize}
\item  [$1)$] strictly convex if the unit sphere $S_1(0) = \{x \in X: ||x||=1\}$ is strictly convex, i.e., the inequality $||x+y|| <2$ holds for 
all $x, y \in S_1(0),  x \ne y ;$ 
 \item [$2)$] uniformly convex if for any given $\epsilon >0$ there exists $\delta = \delta(\epsilon)  >0$ such that for all $x,y \in E$ 
with $\left\|x\right\| \le 1, \left\|y\right\|\le 1,\left\|x-y\right\| = \epsilon$ the inequality $\left\|x+y\right\|\le 2(1-\delta)$ holds;
\item [$3)$] smooth if the limit
\begin{equation}\label{smoonthspa}
\lim\limits_{t \to 0} \frac{\left\|x+ty\right\|-\left\|x\right\|}{t}
\end{equation}
exists for all $x,y \in S_1(0)$;
\item[$4)$] uniformly smooth if the limit $(\ref{smoonthspa})$ exists uniformly for all $x,y\in S_1(0)$.
\end{itemize}
\end{definition}
The modulus of convexity of $E$ is the function $\delta_E:[0,2]\to [0,1]$ defined by
\[
\delta_{E}(\epsilon)=\inf \left\{1-\frac{\left\|x-y\right\|}{2}: \left\|x\right\| = \left\|y\right\|= 1,\left\|x-y\right\|=\epsilon\right\}
\]
for all $\epsilon \in [0,2]$. Note that $E$ is uniformly convex if only if $\delta_{E}(\epsilon)>0$ for all $0<\epsilon\le 2$ and 
$\delta_{E}(0)=0$. Let $p>1$, $E$ is said to be $p$-uniformly convex if there exists some constant $c>0$ such that $\delta_{E}(\epsilon)\ge 
c\epsilon^p$. It is well-known that spaces $L^p, l^p$ and $W_m^p$ are $p$-uniformly convex if $p>2$ and $2$
-uniformly convex if $1< p\le 2$ and a Hilbert space $H$ is uniformly smooth and $2$-uniformly convex.

Let $E$ be a real Banach space with its dual $E^*$. The dual product of $f\in E^*$ and $x\in E$ is denoted by $\left\langle x, f\right\rangle$ 
or $\left\langle f, x\right\rangle$. For the sake of simpicity, the norms of $E$ and $E^*$ are denoted by the same symbol $||.||$. The 
normalized duality mapping $J:E\to 2^{E^*}$ is defined by
\[
J(x)=\left\{f\in E^*: \left\langle f,x \right\rangle=\left\|x\right\|^2=\left\|f\right\|^2\right\}.
\]
The following properties can be found in \cite{C1990}:
\begin{itemize}
\item [i)] If $E$ is a smooth, strictly convex, and reflexive Banach space, then the normalized duality mapping $J:E\to 2^{E^*}$ is 
single-valued, one-to-one, and onto;
\item [ii)] If $E$ is a reflexive and strictly convex Banach space, then $J^{-1}$ is norm to weak $^*$  continuous;
\item [iii)] If $E$ is a uniformly smooth Banach space, then $J$ is uniformly continuous on each bounded subset of $E$;
\item [iv)] A Banach space $E$ is uniformly smooth if and only if $E^*$ is uniformly convex;
\item [v)] Each uniformly convex Banach space $E$ has the Kadec-Klee property, i.e., for any
sequence $\left\{x_n\right\} \subset E$, if $x_n \rightharpoonup x \in E$ and $\left\| x_n \right\|\to \left\|x\right\|$, then $x_n \to x$.
\end{itemize}
\begin{lemma}\label{lem:2-uniform-convex}\cite{ZS2009}
If $E$ is a 2-uniformly convex Banach space, then
\[
||x-y||\le\frac{2}{c^2}||Jx-Jy||,\quad \forall x,y\in E,
\]
where $J$ is the normalized duality mapping on $E$ and $0<c\le1$. 
\end{lemma}
The best constant $\frac{1}{c}$ is called the \textit{$2$-uniform convexity constant} of $E$.\\
Next we assume that $E$ is a smooth, strictly convex, and reflexive Banach space. In the sequel we always use $\phi: E\times E \to [0, \infty)$ 
to denote the Lyapunov functional defined by
\[
\phi(x,y)=\left\|x\right\|^2-2\left\langle x,Jy\right\rangle+\left\|y\right\|^2, \forall x, y \in E.
\]
From the definition of $\phi$, we have
\begin{equation}\label{eq:proPhi}
\left(\left\|x\right\|-\left\|y\right\|\right)^2\le \phi(x,y)\le \left(\left\|x\right\|+\left\|y\right\|\right)^2.
\end{equation}
Moreover, the Lyapunov functional satisfies the identity
\begin{equation}\label{eq:LFP}
\phi(x,y)=\phi(x,z)+\phi(z,y)+2\left\langle z-x,Jy-Jz\right\rangle
\end{equation}
for all $x,y,z\in E$.\\
The generalized projection $\Pi_C:E\to C$ is defined by 
\[
\Pi_C(x)=\arg\min_{y\in C}\phi(x,y).
\]
In what follows, we need the following properties of the functional $\phi$ and the generalized projection $\Pi_C$.
\begin{lemma} \label{lem.proProjec} \cite{A1996}
Let E be a smooth, strictly convex, and reflexive Banach space and $C$ a nonempty closed convex subset of $E$. Then the following 
conclusions hold:
\begin{itemize}
\item [$i)$] $\phi (x,\Pi_C (y))+\phi (\Pi_C (y),y) \le \phi (x,y), \forall x \in C, y \in E$;
\item [$ii)$] if $x\in E, z \in C$, then $z= \Pi_C(x)$ iff $\left\langle z-y;Jx-Jz\right\rangle \ge 0, \forall y \in C$;
\item [$iii)$] $\phi (x,y)=0$ iff $x=y$.
\end{itemize}
\end{lemma}
%%%%%%%%%%%%%%%%%%%%%%%%%%%%%%%%%%%%%%%%%%
\begin{lemma}\cite{KL2008}\label{lem:CloseConvex-Cn}
Let C be a nonempty closed convex subset of a smooth Banach $E$, $x,y,z\in E$ and $\lambda\in [0,1]$. For a given real number $a$, the set
\[
D:=\left\{v\in C: \phi(v,z)\le \lambda\phi(v,x)+(1-\lambda)\phi(v,y)+a\right\}
\]
 is closed and convex.
\end{lemma}
%%%%%%%%%%%%%%%%%%%%%%%%%%%%%%%%%%%%%%%%%%%%%%%%%%%%%%%%
\begin{lemma}\label{lem.prophi}\cite{A1996}
Let $\left\{x_n\right\}$ and $\left\{y_n\right\}$ be two sequences in a uniformly convex and uniformly smooth real Banach space $E$. 
If $\phi(x_n,y_n)\to 0$ and either $\left\{x_n\right\}$ or $\left\{y_n\right\}$ is bounded, then $\left\|x_n-y_n\right\|\to 0$ as $n\to 
\infty$.
\end{lemma}
\begin{lemma}\cite{CKW2010}\label{lem:PK2008}
Let E be a uniformly convex Banach space, $r$ be a positive number and $B_r(0) \subset E $ be a closed ball with center at origin and 
radius $r$. Then, for any given subset 
$\left\{x_1,x_2,\ldots,x_N\right\}\subset B_r(0)$ and for any positive numbers $\lambda_1,\lambda_2,\ldots,\lambda_N$ with 
$\sum_{i=1}^N\lambda_i=1$, there exists a continuous, strictly increasing, and convex function $g:[0,2r)\to [0,\infty)$ with $g(0)=0$ 
such that, for any $i,j\in \left\{1,2,\ldots,N\right\}$ with $i<j$,
$$
\left\|\sum_{k=1}^N\lambda_kx_k\right\|^2\le\sum_{k=1}^N \lambda_k \left\|x_k \right\|^2-\lambda_i\lambda_j g(||x_i-x_j||).
$$
\end{lemma}
\begin{definition}\label{def.monotone}
A mapping $A:E \to E^*$ is called
\begin{itemize}
\item  [$1)$] monotone, if 
$$
\left\langle A(x)-A(y), x-y\right\rangle  \ge 0\quad \forall x,y \in E;
$$
% \item [$2)$] maximal monotone, if it is monotone and its graph is not the right part of the graph of any other monotone operator;
% \item [$3)$] $m$-monotone, if it is monotone and $R(A+\alpha I) = E^*$ for all $\alpha >0,$ where $I$ is the identity operator in 
%$E^*;$
\item [$2)$] uniformly monotone, if there exists a strictly increasing function $\psi : [0, \infty)$ $\to$ $[0, \infty),  \psi(0) = 0,$ 
such that 
\begin {equation}\label{eq:unifaccre}
\left\langle A(x)-A(y), x-y\right\rangle  \ge \psi(||x - y||) \quad \forall x,y \in E;
\end {equation}
\item [$3)$] $\eta$-strongly monotone, if there exists a positive constant $\eta,$ such that in $(\ref{eq:unifaccre})$, $\psi(t) = \eta t^2;$
\item [$4)$] $\alpha$-inverse strongly monotone, if there exists a positive constant $\alpha,$ such that
$$
\left\langle A(x)-A(y), x-y\right\rangle  \ge \alpha||A(x) - A(y)||^2 \quad \forall x,y \in E.
$$
\item [5)] $L$-Lipschitz continuous if there exists a positive constant $L$, such that 
$$
||A(x)-A(y)||\le L||x-y|| \quad \forall x,y\in E.
$$
\end{itemize}
\end{definition}
\begin{rm}
If $A$ is $\alpha$-inverse strongly monotone then it is $\frac{1}{\alpha}$-Lipschitz continuous. If $A$ is $\eta$-strongly monotone 
and $L$-Lipschitz continuous then it is $\frac{\eta}{L^2}$-inverse strongly monotone.
\end{rm}
\begin{lemma}\label{lem:T2000}\cite{T2000}
 Let $C$ be a nonempty, closed convex subset of a Banach space $E$ and $A$ be a monotone, hemicontinuous mapping of $C$ into $E^*$. Then
$$
VI(C,A)=\left\{u\in C:\left\langle v-u,A(v)\right\rangle\ge 0, \quad \forall v\in C\right\}.
$$
\end{lemma}
Let $C$ be a nonempty closed convex subset of  a smooth, strictly convex, and reflexive Banach space $E$,  $T:C\to C$ be a 
mapping. The set
$$
F(T)=\left\{x\in E:Tx=x\right\}
$$
is called the set of fixed points of $T$. A point $p\in C$ is said to be an asymptotic fixed point of $T$ if there exists a sequence 
$\left\{x_n\right\} \subset C$ such that $x_n \rightharpoonup p$ and $\left\|x_n-Tx_n \right\|\to 0$ as 
$n\to +\infty$. The set of all asymptotic fixed points of $T$ will be denoted by 	 $\tilde{F}(T)$.
\begin{definition} \label{def.RNM}
A mapping $T:C\to C$ is called
\begin{itemize}
\item [i)] relatively nonexpansive mapping if $F(T)\ne \emptyset, \tilde{F}(T) = F(T)$, and $$ \phi(p,Tx)\le \phi (p,x), \forall p\in F(T),  \forall x\in C
;$$
\item [ii)] closed if for any sequence $\left\{x_n\right\} \subset C, x_n \to x$ and $Tx_n \to y$, then $Tx=y$;
\item [iii)] quasi $\phi$ - nonexpansive mapping (or hemi-relatively nonexpansive mapping) if $F(T) \ne \emptyset$ and $$ \phi(p,Tx)\le 
\phi (p,x), \forall p\in F(T), \forall x\in C;$$
\item [iv)] quasi $\phi$ - asymptotically nonexpansive if $F(T)\ne \O$ and there exists a sequence 
$\left\{k_n\right\} \subset [1,+\infty)$ with $k_n \to 1$ as $n\to +\infty$ such that $$\phi(p,T^n x)\le k_n\phi (p,x), \forall n\ge 1, \forall p\in 
F(T), \forall x\in C; $$
\item [v)] uniformly $L$ - Lipschitz continuous, if there exists a constant $L>0$ such that
$$
\left\|T^n x-T^n y \right\|\le L\left\|x- y \right\|, \forall n \geq 1 ,  \forall x, y \in C.
$$
\end{itemize}
\end{definition}
The reader is refered to \cite{CKW2010,SWX2009} for examples of closed and asymptotically quasi $\phi$-nonexpansive mappings. 
%%%%%%%%%%%%%%%%%%%%%%%%%%%%%%%%%%%%%%%%%%%%%%%
It has been shown that the class of asymptotically quasi $\phi$-nonexpansive mappings contains properly the class of quasi $\phi$-nonexpansive
 mappings, and the class of quasi $\phi$-nonexpansive mappings contains the class of relatively nonexpansive mappings as a proper subset.
 %%%%%%%%%%%%%%%%%%%%%%%%%%%%%%%%%%%%%%%%%%%%
\begin{lemma} \cite{CKW2010} \label{lem.F-CloseConvex}
Let $E$ be a real uniformly smooth and strictly convex Banach space with 
Kadec-Klee property, and $C$ be a nonempty closed convex subset of $E$. Let $T:C\to C$ be a closed and quasi $\phi$-asymptotically 
nonexpansive mapping with a sequence $\left\{k_n\right\}\subset [1,+\infty), k_n\to 1$. Then $F(T)$ is a closed 
convex subset of $C$.
\end{lemma}
%%%%%%%%%%%%%%%%%%%%%%%%%%%%%%%%%%%%%%%%%%%%%%%%%%%%
Next, for solving the equilibrium problem $(\ref{eq:EP})$, we assume that the bifunction $f$ satisfies the following conditions:
\begin{enumerate}
\item[(A1)] $f(x,x)=0$ for all $x\in C$;
\item[(A2)] $f$ is monotone, i.e., $f(x,y)+f(y,x)\le 0$ for all $x, y \in C$;
\item[(A3)] For all $x,y,z \in C$,
$$ \lim_{t\to 0^+}\sup f(tz+(1-t)x,y) \le f(x,y); $$
\item[(A4)] For all $x\in C$, $f(x,.)$ is convex and lower semicontinuous.
\end{enumerate}
The following results show that in a smooth (uniformly smooth), strictly convex and reflexive Banach space, the regularized equilibrium 
problem has a solution (unique solution), respectively. 
\begin{lemma}\label{ExitenceN0} \cite{TZ2009} Let $C$ be a
closed and convex subset of a smooth, strictly convex and reflexive
Banach space $E$, $f$ be a bifunction from $C\times C$ to
$\mathbb{R}$ satisfying conditions $(A1)$-$(A4)$ and let $r>0$,
$x\in E$. Then there exists $z\in C$ such that
\begin{eqnarray*}
f(z,y)+\frac{1}{r}\langle y-z,Jz-Jx\rangle\geq0, \quad \forall y\in
C.
\end{eqnarray*}
\end{lemma}
\begin{lemma}\label{Tr-ClosedConvex}\cite{TZ2009} Let
$C$ be a closed and convex subset of a uniformly smooth, strictly
convex and reflexive Banach spaces $E$, $f$ be a bifunction from
$C\times C$ to $\mathbb{R}$ satisfying conditions $(A1)$-$(A4)$.
For all $r>0$ and $x\in E$, define the mapping
\begin{eqnarray*}
T_rx=\{z\in C:f(z,y)+\frac{1}{r}\langle y-z,Jz-Jx\rangle\geq0, \quad
\forall y\in C\}.
\end{eqnarray*}
Then the following hold:

{\rm (B1)} $T_r$ is single-valued;

{\rm (B2)} $T_r$ is a firmly nonexpansive-type mapping, i.e., for
all $x, y\in E,$
\begin{eqnarray*}
\langle T_rx-T_ry,JT_rx-JT_ry\rangle\leq\langle T_rx-T_ry,Jx-Jy\rangle;
\end{eqnarray*}

{\rm (B3)} $F(T_r)=\hat{F}(T_r)=EP(f);$

{\rm (B4)} $EP(f)$ is closed and convex and $T_r$ is relatively nonexpansive mapping.
\end{lemma}

\begin{lemma}\label{Tr-QPNM}\cite{TZ2009} Let
$C$ be a closed convex subset of a smooth, strictly convex and reflexive Banach space $E$. Let $f$ be a bifunction from $C\times C$
to $\mathbb{R}$ satisfying $(A1)-(A4)$ and let $r > 0$. Then, for
$x\in E$ and $q\in F(T_r)$,
\begin{eqnarray*}\phi(q,T_rx)+\phi(T_rx,x)\leq\phi(q,x).
\end{eqnarray*}
\end{lemma}
Let $E$ be a real Banach space. Alber \cite{A1996} studied the function $V:E\times E^*\to \mathbb{R}$ defined by
$$
V(x,x^*)=||x||^2-2\left\langle x,x^*\right\rangle+||x^*||^2.
$$
Clearly,  $V(x,x^*)=\phi(x,J^{-1}x^*)$.
\begin{lemma}\label{lem:V-Property}\cite{A1996}
Let E be a refexive, strictly convex and smooth Banach space with $E^*$ as its dual. Then
$$
V(x,x^*)+2\left\langle J^{-1}x-x^*,y^*\right\rangle\le V(x,x^*+y^*), \quad\forall x\in E \quad\mbox{and}\quad \forall x^*,y^*\in E^*.
$$
\end{lemma}
Consider the normal cone $N_C$ to a set $C$ at the point $ x \in C$ defined by
$$
N_C(x)=\left\{x^*\in E^*:\left\langle x-y,x^*\right\rangle\ge 0,\quad \forall y\in C\right\}.
$$
We have the following result.
\begin{lemma}\cite{R1970}\label{lem:N-Normal-Cone}
Let $C$ be a nonempty closed convex subset of a Banach space $E$ and let $A$ be a monotone and hemi-continuous mapping of $C$ into 
$E^*$ with $D(A) = C$. 
Let $Q$ be a mapping defined by:
$$
Q(x)=
\left \{
\begin{array}{ll}
Ax+N_C(x)\quad&\mbox{ if }\quad x\in C,\\
\emptyset  \quad& \mbox{if}\quad x\notin C.
\end{array}
\right.
$$
Then $Q$ is a maximal monotone and $Q^{-1}0=VI(A,C)$.
\end{lemma}
%%%%%%%%%%%%%%%%%%%%%%%%%%%%%%%%%%%%%%%%%
\section{Convergence analysis}
\setcounter{lemma}{0}
\setcounter{theorem}{0}
\setcounter{equation}{0}
\setcounter{remark}{0}
\setcounter{corollary}{0}
Throughout this section, we assume that $C$ is a nonempty closed convex subset of a real uniformly smooth and 2-uniformly convex 
Banach space $E$.  Denote 
$$F=\left(\bigcap_{i=1}^M VI(A_i,C)\right)\bigcap \left(\bigcap_{j=1}^N F(S_j)\right)\bigcap\left(\bigcap_{k=1}^K EP(f_k)\right)$$
and assume that the set $F$ is nonempty.\\ 
We prove convergence theorems for methods $(\ref{eq:VIP-EP-FPP})$ and $(\ref{eq:VIP-EP-FPP1})$ with the control parameter sequences
 satisfying conditions $(\ref{dk:VIP-EP-FPP})$ and $(\ref{dk:VIP-EP-FPP1})$, respectively. We also propose similar parallel hybrid methods for quasi 
$\phi$-nonexpansive mappings, variational inequalities and equilibrium problems. 
\begin{theorem}\label{thm:VIP-EP-FPP}
Let $\left\{A_i\right\}_{i=1}^M$ be a finite family of mappings from $C$ to $E^*$ satisfying conditions {\rm (V1)-(V3)}. 
Let $\left\{f_k\right\}_{k=1}^K:C\times C\to  \mathbb{R}$ be a finite family of bifunctions satisfying conditions
 {\rm (A1)-(A4)}. Let $\left\{S_j\right\}_{j=1}^N:C\to C$ be a finite family of  uniform $L$-Lipschitz continuous and 
quasi-$\phi$-asymptotically nonexpansive mappings with the same sequence $\left\{k_n\right\} \subset [1,+\infty), k_n\to 1$. 
Assume that there exists a positive number $\omega$ such that $F \subset \Omega := \{u \in C: ||u|| \leq \omega \}$. 
If the control parameter sequences  
$\left\{\alpha_n\right\},\left\{\lambda_n\right\},\left\{r_n\right\}$ satisfy condition $(\ref{dk:VIP-EP-FPP})$, 
then the sequence $\left\{x_n\right\}$ generated by $(\ref{eq:VIP-EP-FPP})$ converges strongly to $\Pi_F x_0$.
\end{theorem}

\begin{proof} We divide the proof of Theorem $\ref{thm:VIP-EP-FPP}$ into seven steps.\\
\textbf{Step 1.} Claim that $F, C_n$ are closed convex subsets of $C$. \\
Indeed, since each mapping $S_i$ is uniformly $L$-Lipschitz continuous, it is closed. By Lemmas $\ref{lem:T2000},\ref{lem.F-CloseConvex}$ and 
$\ref{Tr-ClosedConvex}$, $F(S_i),VI(A_j,C)$ and $EP(f_k)$ are closed convex sets, therefore, $\bigcap_{j=1}^N(F(S_j)),$ 
$  \bigcap_{i=1}^MVI(A_i, C)$ and $\bigcap_{k=1}^K EP(f_k)$ are also closed and convex. 
Hence $F$ is a closed and convex subset of $C$. It is obvious that $C_n$ is closed for all $n\ge 0$. We prove the convexity of $C_n$ 
by induction. Clearly, $C_0:=C$ is closed convex. Assume that $C_n$ is closed convex for some $n\ge 0$. From the construction of 
$C_{n+1}$, we find\\
$C_{n+1}=C_n\bigcap \left\{z\in E:\phi(z,\bar{u}_n)\le \phi(z,\bar{z}_n)\le\phi(z,x_n)+\epsilon_n\right\}.$\\
Lemma $\ref{lem:CloseConvex-Cn}$ ensures that $C_{n+1}$ is convex. Thus, $C_n$ is closed convex for all $n\ge 0$. Hence, 
$\Pi_F x_0$ and $x_{n+1}:= \Pi_{C_{n+1}}x_0$ are well-defined.\\
%%%%%%%%%%%%%%%%%%%%%%%%%%%%%%%%%%%%%%%%%%%%%%%%%%
\textbf{Step 2.} Claim that $F\subset C_{n}$ for all $n\ge 0$. \\
By Lemma $\ref{Tr-QPNM}$ and the relative nonexpansiveness of $T_{r_n}$, we obtain $\phi(u,\bar{u}_n)=\phi(u,T_{r_n}\bar{z}_n)\le\phi(u,\bar{z}_n),$
for all $u\in F$. From the convexity of $||.||^2$ and the quasi $\phi$-asymptotical nonexpansiveness of $S_j$, we find
\begin{eqnarray}
\phi(u,\bar{z}_n)&=&\phi \left(u,J^{-1}\left(\alpha_n Jx_n+(1-\alpha_n)JS_{j_n}^n\bar{y}_n\right)\right)\nonumber\\
&=&||u||^2-2\alpha_n\left\langle u,x_n\right\rangle-2(1-\alpha_n)\left\langle u,JS_{j_n}^n\bar{y}_n\right\rangle\nonumber\\
&&+||\alpha_n Jx_n+(1-\alpha_n)JS_{j_n}^n\bar{y}_n||^2\nonumber\\
&\le&||u||^2-2\alpha_n\left\langle u,x_n\right\rangle-2(1-\alpha_n)\left\langle u,JS_{j_n}^n\bar{y}_n\right\rangle\nonumber\\
&&+\alpha_n ||x_n||^2+(1-\alpha_n)||S_{j_n}^n\bar{y}_n||^2\nonumber\\
&=&\alpha_n\phi(u,x_n)+(1-\alpha_n)\phi(u,S_{j_n}^n\bar{y}_n)\nonumber\\
&\le&\alpha_n\phi(u,x_n)+(1-\alpha_n)k_n\phi(u,\bar{y}_n)\label{eq:2}
\end{eqnarray}
for all $u\in F$. By the hypotheses of Theorem $\ref{thm:VIP-EP-FPP}$, Lemmas $\ref{lem:2-uniform-convex},\ref{lem.proProjec},\ref{lem:V-Property}$ and $u\in F$, we have
\begin{eqnarray}
\phi(u,\bar{y}_n)&=&\phi\left(u,\Pi_C\left(J^{-1}(Jx_n-\lambda_n A_{i_n} x_n)\right)\right)\nonumber\\ 
&\le&\phi(u,J^{-1}(Jx_n-\lambda_n A_{i_n} x_n))\nonumber\\ 
&=&V(u,Jx_n-\lambda_n A_{i_n} x_n)\nonumber\\ 
&\le& V(u,Jx_n-\lambda_n A_{i_n} x_n+\lambda_n A_{i_n} x_n)\nonumber\\
&&-2\left\langle J^{-1}\left(Jx_n-\lambda_n A_{i_n} x_n\right)-u,\lambda_n A_{i_n} x_n\right\rangle\nonumber\\
&=&\phi(u,x_n)-2\lambda_n\left\langle J^{-1}\left(Jx_n-\lambda_n A_{i_n} x_n\right)-J^{-1}\left(Jx_n\right),A_{i_n} x_n\right\rangle\nonumber\\
&&-2\lambda_n\left\langle x_n-u,A_{i_n} x_n -A_{i_n}(u) \right\rangle-2\lambda_n\left\langle x_n-u,A_{i_n}u \right\rangle\nonumber\\
&\le& \phi(u,x_n)+\frac{4\lambda_n}{c^2}||Jx_n-\lambda_n A_{i_n} x_n-Jx_n||||A_{i_n} x_n||\nonumber\\
&&-2\lambda_n\alpha||A_{i_n} x_n-A_{i_n}u||^2\nonumber\\
&\le& \phi(u,x_n)+\frac{4\lambda_n^2}{c^2}||A_{i_n} x_n||^2-2\lambda_n\alpha||A_{i_n} x_n-A_{i_n}u||^2\nonumber\\
&\le&\phi(u,x_n)-2a\left(\alpha-\frac{2b}{c^2}\right)||A_{i_n} x_n-A_{i_n}u||^2\nonumber\\
&\le& \phi(u,x_n).\label{eq:3}
\end{eqnarray}
From $(\ref{eq:2}),\left(\ref{eq:3}\right)$ and the estimate $(\ref{eq:proPhi})$, we obtain
\begin{eqnarray}
\phi(u,\bar{z}_n)&\le&\alpha_n\phi(u,x_n)+(1-\alpha_n)k_n\phi(u,x_n)\nonumber\\
&&-2a(1-\alpha_n) \left(\alpha-\frac{2b}{c^2}\right)||A_{i_n} x_n-A_{i_n}u||^2\nonumber\\
&\le&\phi(u,x_n)+(k_n-1)\phi(u,x_n)\nonumber\\
&&-2a(1-\alpha_n) \left(\alpha-\frac{2b}{c^2}\right)||A_{i_n} x_n-A_{i_n}u||^2\nonumber\\
&\le& \phi(u,x_n)+(k_n-1)\left(\omega+||x_n||\right)^2\nonumber\\
&&-2a(1-\alpha_n) \left(\alpha-\frac{2b}{c^2}\right)||A_{i_n} x_n-A_{i_n}u||^2\nonumber\\
&\le& \phi(u,x_n)+\epsilon_n.\label{eq:4}
\end{eqnarray}
Therefore, $F\subset C_n$ for all $n\ge 0$.\\
\textbf{Step 3.} Claim that the sequence $\left\{x_n\right\}$, $\left\{y_n^i\right\}$, $\left\{z_n^j\right\}$ and $\left\{u_n^k\right\}$ converge strongly to $p\in C$.\\
By Lemma $\ref{lem.proProjec}$ and $x_n=\Pi_{C_n}x_0$, we have
$$
\phi(x_n,x_0)\le\phi(u,x_0)-\phi(u,x_n)\le\phi(u,x_0).
$$
for all $u\in F$. Hence $\left\{\phi(x_n,x_0)\right\}$ is bounded. By $(\ref{eq:proPhi})$, $\left\{x_n\right\}$ is bounded, and so are the 
sequences $\left\{\bar{y}_n\right\}$, $\left\{\bar{u}_n\right\}$, and $\left\{\bar{z}_n\right\}$. By the construction of 
$C_n$, $x_{n+1}=\Pi_{C_{n+1}}x_0\in C_{n+1}\subset C_n$. From Lemma $\ref{lem.proProjec}$ and $x_n=\Pi_{C_n}x_0$, we get
$$
\phi(x_{n},x_0)\le\phi(x_{n+1},x_0)-\phi(x_{n+1},x_n)\le\phi(x_{n+1},x_0).
$$
Therefore, the sequence $\left\{\phi(x_{n},x_0)\right\}$ is nondecreasing, hence  it has a finite limit. Note that, for all 
$m\ge n$, $x_m\in C_m\subset C_n,$ and by Lemma $\ref{lem.proProjec}$ we obtain
\begin{equation}\label{eq:7}
\phi(x_{m},x_n)\le\phi(x_{m},x_0)-\phi(x_{n},x_0)\to 0
\end{equation}
as $m,n\to\infty$. From $(\ref{eq:7})$ and Lemma $\ref{lem.prophi}$ we have $||x_n-x_m||\to 0$. This shows that 
$\left\{x_n\right\}\subset C$ is a
Cauchy sequence. Since $E$ is complete and $C$ is closed convex subset of $E$, $\left\{x_n\right\}$ converges strongly to $p\in C$. 
From $\left(\ref{eq:7}\right)$, $\phi(x_{n+1},x_n)\to 0$ as $n\to\infty$. Taking into account that $x_{n+1}\in C_{n+1}$, we find
\begin{equation}\label{eq:8}
\phi(x_{n+1},\bar{u}_n)\le \phi(x_{n+1},\bar{z}_n)\le\phi(x_{n+1},x_n)+\epsilon_n
\end{equation}
Since $\left\{x_n\right\}$ is bounded, we can put $M=\sup\left\{||x_n||:n=0,1,2,\ldots\right\},$ hence
\begin{equation}\label{eq:9}
\epsilon_n:= (k_n-1) (\omega + ||x_n||)^2 \le (k_n-1)(\omega +M)^2\to 0.
\end{equation}
By $(\ref{eq:8}),(\ref{eq:9})$ and $\phi(x_{n+1},x_n)\to 0$, we find that 
\begin{equation}\label{eq:10}
\lim_{n\to\infty}\phi(x_{n+1},\bar{u}_n)=\lim_{n\to\infty}\phi(x_{n+1},\bar{z}_n)=\lim_{n\to\infty}\phi(x_{n+1},x_n)=0.
\end{equation}
Therefore, from Lemma $\ref{lem.prophi}$, 
$$\lim_{n\to\infty}||x_{n+1}-\bar{u}_n||=\lim_{n\to\infty}||x_{n+1}-\bar{z}_n||=\lim_{n\to\infty}||x_{n+1}-x_n||=0.$$
This together with $||x_{n+1}-x_n||\to 0$ implies that 
$$
\lim_{n\to\infty}||x_{n}-\bar{u}_n||=\lim_{n\to\infty}||x_{n}-\bar{z}_n||=0.
$$
By the definitions of $j_n$ and $k_n$, we obtain
\begin{equation}\label{eq:11}
\lim_{n\to\infty}||x_{n}-u^k_n||=\lim_{n\to\infty}||x_{n}-z^j_n||=0.
\end{equation}
for all $1\le k\le K$ and $1\le j\le N$. Hence 
\begin{equation}\label{eq:12}
\lim_{n\to\infty}x_n=\lim_{n\to\infty}u^k_n=\lim_{n\to\infty}z^j_n=p
\end{equation}
for all $1\le k\le K$ and $1\le j\le N$.
By the hypotheses of Theorem $\ref{thm:VIP-EP-FPP}$, Lemmas $\ref{lem:2-uniform-convex},\ref{lem.proProjec}$ and $\ref{lem:V-Property}$, we also have
\begin{eqnarray}
\phi(x_n,\bar{y}_n)&=&\phi\left(x_n,\Pi_C\left(J^{-1}(Jx_n-\lambda_n A_{i_n} x_n)\right)\right)\nonumber\\ 
&\le&\phi(x_n,J^{-1}(Jx_n-\lambda_n A_{i_n} x_n))\nonumber\\ 
&=&V(x_n,Jx_n-\lambda_n A_{i_n} x_n)\nonumber\\ 
&\le& V(x_n,Jx_n-\lambda_n A_{i_n} x_n+\lambda_n A_{i_n} x_n)\nonumber\\
&&-2\left\langle J^{-1}\left(Jx_n-\lambda_n A_{i_n} x_n\right)-x_n,\lambda_n A_{i_n} x_n\right\rangle\nonumber\\
&=&-2\lambda_n \left\langle J^{-1}\left(Jx_n-\lambda_n A_{i_n}x_n\right)-J^{-1}Jx_n,A_{i_n} x_n\right\rangle\nonumber\\
&\le& \frac{4\lambda_n}{c^2}||Jx_n-\lambda_n A_{i_n}x_n-Jx_n||||A_{i_n} x_n||\nonumber\\
&\le& \frac{4\lambda_n^2}{c^2}||A_{i_n} x_n||^2\nonumber\\
&\le& \frac{4b^2}{c^2}||A_{i_n} x_n-A_{i_n}u||^2 \label{eq:13}
\end{eqnarray}
for all $u\in \bigcap_{i=1}^M VI(A_i,C)$. From $(\ref{eq:4})$, we obtain
\begin{eqnarray}
&&2(1-\alpha_n)a \left(\alpha-\frac{2b}{c^2}\right)||A_{i_n} x_n-A_{i_n}u||^2\le\left(\phi(u,x_n)-\phi(u,\bar{z}_n)\right)+\epsilon_n\nonumber\\
&&\quad\quad=2\left\langle u,J\bar{z}_n-Jx_n\right\rangle+(||x_n||^2-||\bar{z}_n||^2)+\epsilon_n\nonumber\\
&&\quad\quad\le 2||u||||J\bar{z}_n-Jx_n||+||x_n-\bar{z}_n||(||x_n||+||\bar{z}_n||)+\epsilon_n.\label{eq:14}
\end{eqnarray}
Using the fact that $||x_n-\bar{z}_n||\to 0$ and $J$ is uniformly continuous on each bounded set, we can conclude that 
$||J\bar{z}_n-Jx_n||\to 0$ as $n\to\infty$. This together with $(\ref{eq:14}),$ and the relations $\lim\sup_{n\to\infty}\alpha_n<1$ and $\epsilon_n\to 0$ implies that 
\begin{equation}\label{eq:15}
\lim_{n\to\infty}||A_{i_n} x_n-A_{i_n}u||=0.
\end{equation}
From $(\ref{eq:13})$ and $(\ref{eq:15})$, we obtain
$$
\lim_{n\to\infty}\phi(x_n,\bar{y}_n)=0.
$$
Therefore $\lim_{n\to\infty}||x_n-\bar{y}_n||=0$. By the definition of $i_n$, we get\\
 $\lim_{n\to\infty}||x_n-y^i_n||=0$. Hence, 
\begin{equation}\label{eq:17}
\lim_{n\to\infty}y_n^i=p
\end{equation}
for all $1\le i\le M$.\\
\textbf{Step 4.} Claim that $p\in \bigcap_{j=1}^n F(S_j)$.\\
The relation $z_n^j=J^{-1}\left(\alpha_n Jx_n+(1-\alpha_n)JS_j^n\bar{y}_n\right)$ implies that $Jz_n^j=\alpha_n Jx_n+(1-\alpha_n)JS_j^n\bar{y}_n$. Therefore,
\begin{equation}\label{eq:18}
||Jx_n-Jz_n^j||=(1-\alpha_n)||Jx_n-JS_j^n\bar{y}_n||.
\end{equation}
Since $||x_n-z_n^j||\to 0$ and $J$ is uniformly continuous on each bounded subset of $E$, $||Jx_n-Jz_n^j||\to 0$ as $n\to\infty$. This together with $(\ref{eq:18})$ and $\lim\sup_{n\to\infty}\alpha_n<1$ implies that 
$$
\lim_{n\to\infty}||Jx_n-JS_j^n\bar{y}_n||=0.
$$
Therefore, 
\begin{equation}\label{eq:20}
\lim_{n\to\infty}||x_n-S_j^n\bar{y}_n||=0.
\end{equation}
Since $\lim_{n\to\infty}||x_n-\bar{y}_n||=0$, $\lim_{n\to\infty}||\bar{y}_n-S_j^n\bar{y}_n||=0$, hence
\begin{equation}\label{eq:21}
\lim_{n\to\infty}S_j^n\bar{y}_n=p.
\end{equation}
Further, 
\begin{eqnarray*}
\left\|S^{n+1}_j \bar{y}_n-S^{n}_j \bar{y}_n\right\|&\le& \left\|S^{n+1}_j \bar{y}_n-S^{n+1}_j \bar{y}_{n+1}\right\| +\left\|S^{n+1}_j \bar{y}_{n+1} -\bar{y}_{n+1}
\right\|\\
&&+\left\|\bar{y}_{n+1}-\bar{y}_{n}\right\|+\left\|\bar{y}_{n}-S^{n}_j \bar{y}_{n}\right\|\\ 
 &\le& \left(L+1\right) \left\|\bar{y}_{n+1}-\bar{y}_{n}\right\| +\left\|S^{n+1}_j \bar{y}_{n+1} -\bar{y}_{n+1}\right\|\\
&&+\left\|\bar{y}_{n}-S^{n}_j \bar{y}_{n}\right\| \to 0,
\end{eqnarray*}
therefore
\begin{equation}\label{eq:22}
\lim_{n\to\infty}S_j^{n+1}\bar{y}_n=\lim_{n\to\infty}S_j\left(S_j^{n}\bar{y}_n\right)=p.
\end{equation}
From $(\ref{eq:21})$,$(\ref{eq:22})$ and the closedness of $S_j$, we obtain $p\in F(S_j)$ for all $1\le j\le N$. Hence $p\in \bigcap_{j=1}^N F(S_j)$.\\
\textbf{Step 5.} Claim that $p\in \bigcap_{i=1}^M VI(A_i,C)$.\\
Lemma $\ref{lem:N-Normal-Cone}$ ensures that the mapping
$$
Q_i(x)=\left \{
\begin{array}{ll}
&A_ix+N_C(x)\quad\mbox{ if }\quad x\in C,\\
&\O \quad\quad\quad\quad\quad\quad \mbox{if}\quad x\notin C.
\end{array}
\right.
$$
is maximal monotone, where $N_C(x)$ is the normal cone to  $C$ at $x\in C$. For all $(x,y)$ in the graph of $Q_i$, i.e., $(x,y)\in G(Q_i)$, we have $y-A_i(x)\in N_C(x)$. By the definition of $N_C(x)$, we find that 
$$\left\langle x-z,y-A_i(x)\right\rangle\ge 0$$
for all $z\in C$. Since $y_n^i\in C$,
$$
\left\langle x-y_n^i,y-A_i(x)\right\rangle\ge 0.
$$
Therefore,
\begin{equation}\label{eq:24}
\left\langle x-y_n^i,y\right\rangle\ge\left\langle x-y_n^i,A_i(x)\right\rangle.
\end{equation}
Taking into account $y_n^i=\Pi_C\left(J^{-1}(Jx_n-\lambda_n A_i x_n)\right)$ and Lemma $\ref{lem.proProjec}$, we get
\begin{equation}\label{eq:25}
\left\langle x-y_n^i,Jy_n^i-Jx_n+\lambda_n A_i x_n\right\rangle\ge 0.
\end{equation}
Therefore, from $(\ref{eq:24}),(\ref{eq:25})$ and the monotonicity of $A_i$, we find that
\begin{eqnarray}
\left\langle x-y_n^i,y\right\rangle&\ge&\left\langle x-y_n^i,A_i(x)\right\rangle\nonumber\\ 
&=& \left\langle x-y_n^i,A_i(x)-A_i(y_n^i)\right\rangle+\left\langle x-y_n^i,A_i(y_n^i)-A_i(x_n)\right\rangle\nonumber\\
&&+\left\langle x-y_n^i,A_i(x_n)\right\rangle\nonumber\\
&\ge& \left\langle x-y_n^i,A_i(y_n^i)-A_i(x_n)\right\rangle+\left\langle x-y_n^i,\frac{Jx_n-Jy_n^i}{\lambda_n}\right\rangle. \label{eq:26}
\end{eqnarray}
Since $||x_n-y_n^i||\to 0$ and $J$ is uniform continuous on each bounded set, $||Jx_n-Jy_n^i||\to 0$. By $\lambda_n\ge a>0$, we obtain
\begin{equation}\label{eq:27}
\lim_{n\to\infty}\frac{Jx_n-Jy_n^i}{\lambda_n}=0.
\end{equation}
Since $A_i$ is $\alpha$-inverse strongly monotone, $A_i$ is $\frac{1}{\alpha}$-Lipschitz continuous. This together with $||x_n-y_n^i||\to 0$ implies that 
\begin{equation}\label{eq:28}
\lim_{n\to\infty}||A_i(y_n^i)-A_i(x_n)||=0.
\end{equation}
From $(\ref{eq:26})$, $(\ref{eq:27})$,$(\ref{eq:28})$, and $y_n^i\to p$, we obtain $\left\langle x-p,y\right\rangle\ge 0$ for all $(x,y)\in G(Q_i)$. Therefore $p\in Q_i^{-1}0=VI(A_i,C)$ for all $1\le i\le M$. Hence, $p\in \bigcap_{i=1}^M VI(A_i,C)$.\\
\textbf{Step 6.} Claim that $p\in \bigcap_{k=1}^K EP(f_k)$.\\
Since  $\lim_{n\to\infty}\left\|u_n^k-\bar{z}_n\right\|=0$ and $J$ is uniformly continuous on every bounded subset of $E$, we have 
$$
\lim_{n\to\infty}\left\|Ju_n^k-J\bar{z}_n\right\|=0.
$$
This together with $r_n\ge d>0$ implies that
\begin{equation}\label{eq:2.12}
\lim_{n\to\infty}\frac{\left\|Ju_n^k-J\bar{z}_n\right\|}{r_n}=0.
\end{equation}
We have $u_n^k=T_{r_n}^k\bar{z}_n$, and
\begin{equation}\label{eq:2.13}
f_k(u_n^k,y)+\frac{1}{r_n}\left\langle y-u_n^k,Ju_n^k-J\bar{z}_n\right\rangle \ge 0 \quad \forall y\in C.
\end{equation}
From $(\ref{eq:2.13})$ and condition $(A2)$, we get
\begin{equation}\label{eq:2.14}
\frac{1}{r_n}\left\langle y-u_n^k,Ju_n^k-J\bar{z}_n\right\rangle \ge -f_k(u_n^k,y)\ge f_k(y,u_n^k) \quad \forall y\in C.
\end{equation}
Letting $n\to\infty$, by $(\ref{eq:2.12}),(\ref{eq:2.14})$ and $(A4)$, we obtain
\begin{equation}\label{eq:2.15}
f_k(y,p)\le 0,\, \forall y\in C.
\end{equation}
Putting $y_t=ty+(1-t)p$, where $0<t\le 1$ and $y\in C$, we get $y_t \in C$. Hence, for sufficiently small
 $t$, from $(A3)$ and $(\ref{eq:2.15})$, we have 
$$
f_k(y_t,p)=f_k(ty+(1-t)p,p)\le 0.
$$
By the properties $(A1),(A4)$, we find
\begin{eqnarray*}
0&=&f_k(y_t,y_t)\\ 
&=&f_k(y_t,ty+(1-t)p) \\
&\le& tf_k(y_t,y)+(1-t)f_k(y_t,p)\\
&\le& tf_k(y_t,y)
\end{eqnarray*}
Dividing both sides of the last inequality by $t>0$, we obtain $f_k(y_t,y)\ge 0$ for all $y\in C$, i.e.,
$$ f_k(ty+(1-t)p,y)\ge 0,\, \forall y\in C. $$
Passing $t\to 0^+$, from $(A3)$, we get $f_k(p,y)\ge 0,\, \forall y\in C$ and $1\le k\le K$, i.e., $p\in \bigcap_{k=1}^K EP(f_k)$.\\
\textbf{Step 7.} Claim that the sequence $\left\{x_n\right\}$ converges strongly to $\Pi_F x_0$.\\
Indeed, since $x^{\dagger}:=\Pi_F(x_0)\in F \subset C_n,\,x_n=\Pi_{C_n}(x_0)$ from Lemma $\ref{lem.proProjec}$, we have
\begin{equation}\label{eq:29}
\phi(x_n,x_0)\le \phi(x^\dagger,x_0)-\phi(x^\dagger,x_n)\le \phi(x^\dagger,x_0).
\end{equation}
Therefore, 
\begin{eqnarray*}
\phi(x^\dagger,x_0)&\ge& \lim\limits_{n\to \infty}\phi(x_n,x_0)=\lim\limits_{n\to \infty}\left\{\left\|x_n\right\|^2-2\left\langle x_n,Jx_0\right\rangle+\left\|x_0\right\|^2\right\}\\ 
&=& \left\|p\right\|^2-2\left\langle p,Jx_0\right\rangle+\left\|x_0\right\|^2\\
&=&\phi(p,x_0).
\end{eqnarray*}
From the definition of $x^\dagger$, it follows that $p=x^\dagger$. The proof of Theorem $\ref{thm:VIP-EP-FPP}$ is complete.
\end{proof}
\begin{remark}
Assume that $\left\{A_i\right\}_{i=1}^M$ is a finite family of $\eta$-strongly monotone and $L$-Lipschitz continuous mappings. Then
each $A_i$ is $\frac{\eta}{L}$-inverse strongly monotone and $VI(A_i,C)=A_i^{-1}0$. Hence, $||A_i x||\le ||A_ix-A_iu||$ for all $x\in C$ 
and $u\in VI(A_i,C)$. Thus, all the conditions (V1)-(V3) for the variational inequalities $VI(A_i,C)$ hold.
\end{remark}
\begin{theorem}\label{thm:VIP-EP-FPP1}
Let $\left\{A_i\right\}_{i=1}^M$ be a finite family of mappings from $C$ to $E^*$ satisfying conditions {\rm (V1)-(V3)}. Let 
$\left\{f_k\right\}_{k=1}^K:C\times C\to  \mathbb{R}$ be a finite family of bifunctions satisfying conditions 
{\rm (A1)-(A4)}. Let $\left\{S_j\right\}_{j=1}^N:C\to C$ be a finite family of uniform $L$-Lipschitz continuous and 
quasi-$\phi$-asymptotically nonexpansive mappings with the same sequence $\left\{k_n\right\} \subset [1,+\infty), k_n\to 1$. 
Assume that $F$ is a subset of $\Omega$, and suppose that the control parameter sequences  
$\left\{\alpha_n\right\},\left\{\lambda_n\right\},\left\{r_n\right\}$ satisfy condition $(\ref{dk:VIP-EP-FPP1})$. 
Then the sequence $\left\{x_n\right\}$ generated by method $(\ref{eq:VIP-EP-FPP1})$ converges strongly to $\Pi_F x_0$.
\end{theorem}
\begin{proof}
Arguing similarly as in Step 1 of the proof of Theorem $\ref{thm:VIP-EP-FPP}$, we conclude that $F, C_n$ are closed convex for all 
$n\ge 0$. Now we show that $F\subset C_n$ for all $n\ge 0$. For all $u\in F$, by Lemma $\ref{lem:PK2008}$ and the convexity of $||.||^2$ 
we obtain
\begin{eqnarray}
\phi(u,z_n)&=&\phi \left(u,J^{-1}\left(\alpha_{n,0} Jx_n+\sum_{l=1}^N \alpha_{n,l} JS_l^n\bar{y}_n\right)\right)\nonumber\\
&=&||u||^2-2\alpha_{n,0}\left\langle u,x_n\right\rangle-2\sum_{l=1}^N \alpha_{n,l}\left\langle u,S_l^n\bar{y}_n\right\rangle\nonumber\\
&&+||\alpha_{n,0} Jx_n+\sum_{l=1}^N \alpha_{n,l} JS_l^n\bar{y}_n||^2\nonumber\\
&\le&||u||^2-2\alpha_{n,0}\left\langle u,x_n\right\rangle-2\sum_{l=1}^N \alpha_{n,l}\left\langle u,S_l^n\bar{y}_n\right\rangle+\alpha_{n,0}||x_n||^2\nonumber\\
&&+\sum_{l=1}^N \alpha_{n,l}||S_l^n\bar{y}_n||^2-\alpha_{n,0}\alpha_{n,j}g\left(||Jx_n-JS_j^n\bar{y}_n||\right)\nonumber\\
&\le&\alpha_{n,0}\phi(u,x_n)+\sum_{l=1}^N \alpha_{n,l}\phi(u,S_{l}^n\bar{y}_n)-\alpha_{n,0}\alpha_{n,j}g\left(||Jx_n-JS_j^n\bar{y}_n||\right)\nonumber\\
&\le&\alpha_{n,0}\phi(u,x_n)+\sum_{l=1}^N \alpha_{n,l}k_n\phi(u,\bar{y}_n)-\alpha_{n,0}\alpha_{n,j}g\left(||Jx_n-JS_j^n\bar{y}_n||\right).\nonumber\\\label{eq:30}
\end{eqnarray}
From $(\ref{eq:3})$, we get
\begin{equation}\label{eq:31}
\phi(u,\bar{y}_n)\le\phi(u,x_n)-2a\left(\alpha-\frac{2b}{c^2}\right)||A_{i_n} x_n-A_{i_n}u||^2. 
\end{equation}
Using  $(\ref{eq:30})$, $(\ref{eq:31})$ and the estimate $(\ref{eq:proPhi})$, we find
\begin{eqnarray}
\phi(u,\bar{u}_n)&=&\phi(u,T_{r_n}^{k_n}z_n)\nonumber\\
&\le&\phi(u,z_n)\nonumber\\
&\le& \phi(u,x_n)+\sum_{l=1}^N \alpha_{n,l}(k_n-1)\phi(u,x_n)-\alpha_{n,0}\alpha_{n,j}g\left(||Jx_n-JS_j^n\bar{y}_n||\right)\nonumber\\ 
&&-2\sum_{l=1}^N \alpha_{n,l}a\left(\alpha-\frac{2b}{c^2}\right)||A_{i_n} x_n-A_{i_n}u||^2\nonumber\\
& \le& \phi(u,x_n)+(k_n-1)\left(\omega+||x_n||\right)^2-\alpha_{n,0}\alpha_{n,j}g\left(||Jx_n-JS_j^n\bar{y}_n||\right)\nonumber\\ 
&&-2\sum_{l=1}^N \alpha_{n,l}a\left(\alpha-\frac{2b}{c^2}\right)||A_{i_n} x_n-A_{i_n}u||^2\nonumber\\
& \le& \phi(u,x_n)+\epsilon_n. \label{eq:32}
\end{eqnarray}
Therefore
$$
\phi(u,\bar{u}_n)\le \phi(u,x_n)+\epsilon_n
$$
for all $u\in F$. This implies that $F\subset C_n$ for all $n\ge 0$. Using $(\ref{eq:32})$ and arguing similarly as in Steps 3, 5, 6 of Theorem $\ref{thm:VIP-EP-FPP}$, we obtain
$$
\lim_{n\to\infty}\bar{u}_n=\lim_{n\to\infty}u_n^k=\lim_{n\to\infty}\bar{y}_n=\lim_{n\to\infty}y^i_n=\lim_{n\to\infty}x_n=p\in C,
$$
and $p\in \left(\bigcap_{k=1}^K EP(f_k)\right)\bigcap \left(\bigcap_{i=1}^M VI(A_i,C)\right)$.\\
 Next, we show that $p\in$ $\bigcap_{j=1}^N F(S_j)$. Indeed, from $(\ref{eq:32})$, we have
\begin{equation}\label{eq:34}
\alpha_{n,0}\alpha_{n,j}g\left(||Jx_n-JS_j^n\bar{y}_n||\right)\le (\phi(u,x_n)-\phi(u,\bar{u}_n))+\epsilon_n.
\end{equation}
Since $||x_n-\bar{u}_n||\to 0$, $|\phi(u,x_n)-\phi(u,\bar{u}_n)|\to 0$ as $n\to\infty$. This together with $(\ref{eq:34})$
and the facts that $\epsilon_n\to 0$ and $ \lim\inf_{n\to \infty}\alpha_{n,0}\alpha_{n,j}>0$ implies that
$$
\lim_{n\to \infty}g\left(||Jx_n-JS_j^n\bar{y}_n||\right)=0.
$$
By Lemma $\ref{lem:PK2008}$, we get
$$
\lim_{n\to \infty}||Jx_n-JS_j^n\bar{y}_n||=0.
$$
Since $J$ is uniformly continuous on each bounded subset of $E$, we conclude that
$$
\lim_{n\to \infty}||x_n-S_j^n\bar{y}_n||=0.
$$
Using the last equality and  a similar argument for proving relations $(\ref{eq:20})$, $(\ref{eq:21})$, $(\ref{eq:22})$, and acting as in Step 7 of the 
proof of Theorem $\ref{thm:VIP-EP-FPP}$, we obtain $p\in \bigcap_{j=1}^N F(S_j)$ and $p=x^\dagger=\Pi_F x_0$. The proof of Theorem $\ref{thm:VIP-EP-FPP1}$ is complete.
\end{proof}
%%%%%%%%%%%%%%%%%%%%%%%%%%%%%%%%%%%%%%%%%%%%%%%%%
Next, we consider two parallel hybrid methods for solving variational inequalities, equilibrium problems and quasi $\phi$-nonexpansive 
mappings, when the boundedness of the solution set $F$ and the uniform Lipschitz continuity of $S_i$ are not assumed.
\begin{theorem}\label{thm:VIP-EP-FPP2}
Assume that $\left\{A_i\right\}_{i=1}^M,\left\{f_k\right\}_{k=1}^K,\left\{\alpha_n\right\},\left\{r_n\right\}$ and $\left\{\lambda_n\right\}$ 
satisfy all conditions of Theorem $\ref{thm:VIP-EP-FPP}$ and $\left\{S_j\right\}_{j=1}^N$ is a finite family of closed and quasi 
$\phi$-nonexpansive mappings. In addition, suppose that the solution set $F$ is nonempty. For an intitial point $x_0\in C$, define the sequence $\left\{x_n\right\}$ as follows:
\begin{equation}\label{eq:VIP-EP-FPP2}
\left \{
\begin{array}{ll}
&y_n^i=\Pi_C\left(J^{-1}(Jx_n-\lambda_n A_i x_n)\right), i=1,2,\ldots M, \\
&i_n=\arg\max\left\{||y_n^i-x_n||:i=1,2,\ldots M.\right\}, \bar{y}_n=y_n^{i_n},\\
&z_n^j=J^{-1}\left(\alpha_n Jx_n+(1-\alpha_n)JS_j\bar{y}_n\right),j=1,2,\ldots N,\\
&j_n=\arg\max\left\{||z_n^j-x_n||:j=1,2,\ldots N\right\}, \bar{z}_n=z_n^{j_n},\\
&u_n^k=T_{r_n}^k \bar{z}_n,k=1,2,\ldots K, \\
&k_n=\arg\max\left\{||u_n^k-x_n||:k=1,2,\ldots K\right\}, \bar{u}_n=u_n^{k_n},\\
&C_{n+1}=\left\{z\in C_n:\phi(z,\bar{u}_n)\le \phi(z,\bar{z}_n)\le\phi(z,x_n)\right\},\\
&x_{n+1}=\Pi_{C_{n+1}}x_0,n\ge 0.
\end{array}
\right.
\end{equation}
Then the sequence $\left\{x_n\right\}$ converges strongly to $\Pi_F x_0$.
\end{theorem}
\begin{proof}
Since $S_i$ is a closed and quasi $\phi$-nonexpansive mapping, it is closed and quasi $\phi$-asymptotically nonexpansive mapping with 
$k_n=1$ for all $n\ge 0$. Hence, $\epsilon_n=0$ by definition. Arguing similarly as in the proof of Theorem $\ref{thm:VIP-EP-FPP}$, we 
come to the desired conclusion.
\end{proof}
%%%%%%%%%%%%%%%%%%%%%%%%%%%%%%%%%%%%%%%
\begin{theorem}\label{thm:VIP-EP-FPP4}
Assume that $\left\{A_i\right\}_{i=1}^M,\left\{f_k\right\}_{k=1}^K,\left\{r_n\right\},\left\{\alpha_{n,j}\right\}$ and $\left\{\lambda_n\right\}$ 
satisfy all conditions of Theorem $\ref{thm:VIP-EP-FPP1}$ and $\left\{S_j\right\}_{j=1}^N$ is a finite family of closed and quasi 
$\phi$-nonexpansive mappings. In addition, suppose that the solution set $F$ is nonempty. For an initial approximation $x_0\in C$, let the sequence 
$\left\{x_n\right\}$ be defined by
\begin{equation}\label{eq:VIP-EP-FPP2}
\left \{
\begin{array}{ll}
&y_n^i=\Pi_C\left(J^{-1}(Jx_n-\lambda_n A_i x_n)\right), i=1,2,\ldots M, \\
&i_n=\arg\max\left\{||y_n^i-x_n||:i=1,2,\ldots M.\right\}, \bar{y}_n=y_n^{i_n},\\
&z_n=J^{-1}\left(\alpha_{n,0} Jx_n+\sum_{j=1}^N \alpha_{n,j} JS_j\bar{y}_n\right),\\
&u_n^k=T_{r_n}^k z_n,k=1,2,\ldots K, \\
&k_n=\arg\max\left\{||u_n^k-x_n||:k=1,2,\ldots K\right\}, \bar{u}_n=u_n^{k_n},\\
&C_{n+1}=\left\{z\in C_n:\phi(z,\bar{u}_n)\le\phi(z,x_n)\right\},\\
&x_{n+1}=\Pi_{C_{n+1}}x_0,n\ge 0.
\end{array}
\right.
\end{equation}
Then the sequence $\left\{x_n\right\}$ converges strongly to $\Pi_F x_0$.
\end{theorem}
\begin{proof}
The proof  is similar to that of Theorem $\ref{thm:VIP-EP-FPP1}$ for $S_i$ being closed and quasi $\phi$- asymptotically nonexpansive mapping with $k_n=1$ for all $n\ge 0$.
\end{proof}
\section{A parallel iterative method for quasi $\phi$-nonexpansive mappings and variational inequalities}
\setcounter{theorem}{0}
\setcounter{equation}{0}
\setcounter{remark}{0}
\setcounter{corollary}{0}
In 2004, using Mann's iteration, Matsushita and Takahashi \cite{MT2004} proposed the following scheme for finding a fixed point of a 
relatively nonexpansive mapping $T$:
\begin{equation} \label{eq:MT2004}
x_{n+1}=\Pi_C J^{-1}\left(\alpha_nJx_n+(1-\alpha_n)JTx_n\right),\quad n=0,1,2,\ldots,
\end{equation}
where $x_0\in C$ is given.
They proved that if the interior of $F(T)$ is nonempty then the sequence $\left\{x_n\right\}$ generated by $(\ref{eq:MT2004})$ converges 
strongly to some point in $F(T)$. Recently, using Halpern's and Ishikawa's iterative processes, Zhang, Li, and Liu \cite{ZLL2011} have 
proposed modified iterative algorithms of $(\ref{eq:MT2004})$ for a relatively nonexpansive mapping. \\
In this section, employing the ideas of Matsushita and Takahashi \cite{MT2004} and Anh and Chung \cite{AC2013}, we propose a parallel 
hybrid iterative algorithm for  finite families of closed and quasi $\phi$- nonexpansive mappings $\{S_j\}_{j=1}^N$ and variational 
inequalities $\{VI(A_i, C)\}_{i=1}^M$:
\begin{equation}\label{eq:VIP-EP-FPP3}
\left \{
\begin{array}{ll}
&x_0\in C\quad \mbox{chosen arbitrarily},\\
&y_n^i=\Pi_C\left(J^{-1}(Jx_n-\lambda_n A_i x_n)\right), i=1,2,\ldots M, \\
&i_n=\arg\max\left\{||y_n^i-x_n||:i=1,2,\ldots M.\right\}, \bar{y}_n=y_n^{i_n},\\
&z_n^j=J^{-1}\left(\alpha_n Jx_n+(1-\alpha_n)JS_j\bar{y}_n\right),j=1,2,\ldots N,\\
&j_n=\arg\max\left\{||z_n^j-x_n||:j=1,2,\ldots N\right\}, \bar{z}_n=z_n^{j_n},\\
&x_{n+1}=\Pi_{C}\bar{z}_n,n\ge 0,
\end{array}
\right.
\end{equation}
where, $\left\{\alpha_n\right\}\subset [0,1]$, such that $\lim_{n\to\infty}\alpha_n=0$.
\begin{remark}
Method $(\ref{eq:VIP-EP-FPP3})$ can be employed for a finite family of relatively nonexpansive mappings without the assumption on their 
closedeness. 
\end{remark}
\begin{remark} Method $(\ref{eq:VIP-EP-FPP3})$ modifies the corresponding method $(\ref{eq:MT2004})$ in the following aspects:
\begin{itemize}
\item A relatively nonexpansive mapping $T$ is replaced with a finite family of quasi $\phi$-nonexpansive mappings, where the restriction 
$F(S_j)=\widehat{F}(S_j)$ is not required.
\item A parallel hybrid method for finite families of closed and quasi $\phi$- nonexpansive mappings and variational inequalities is 
considered instead of an iterative method for a relatively nonexpansive mapping.
\end{itemize}
\end{remark}
\begin{theorem}\label{thm:VIP-EP-FPP3}
Let $E$ be a real uniformly smooth and 2-uniformly convex Banach space with dual space $E^*$ and $C$ be a nonempty closed convex 
subset of $E$. Assume that $\left\{A_i\right\}_{i=1}^M$ is a finite family of mappings satisfying conditions 
{\rm (V1)-(V3)}, $\left\{S_j\right\}_{j=1}^N$ is a finite family of closed and quasi $\phi$-nonexpansive mappings, and 
$\left\{\alpha_n\right\}\subset [0,1]$ satisfies $\lim_{n\to\infty}\alpha_n=0$, $\lambda_n\in [a,b]$ for some $a,b\in (0,\alpha c^2 /2)$. In addition, suppose that the interior of $F=\left(\bigcap_{i=1}^M VI(A_i,C)\right)$ $\bigcap$ $\left(\bigcap_{j=1}^N F(S_j)\right)$ is nonempty. 
Then the sequence $\left\{x_n\right\}$ generated by $(\ref{eq:VIP-EP-FPP3})$ converges strongly to some point $u\in F$. Moreover, $u=\lim_{n\to\infty}\Pi_F x_n.$
\end{theorem}
\begin{proof}
By Lemma $\ref{lem.F-CloseConvex}$, the subset $F$ is closed and convex, hence the generalized projections $\Pi_F, \Pi_C$ are 
well-defined. We now show that the sequence $\left\{x_n\right\}$ is bounded. Indeed, for every $u\in F$, from Lemma $\ref{lem.proProjec}$ and the convexity of $\left\|.\right\|^2$, we have
\begin{eqnarray*}
\phi(u,x_{n+1})&=&\phi(u,\Pi_C \bar{z}_n)\\
&\le&\phi(u,\bar{z}_n)\\
&=& \left\|u\right\|^2-2\left\langle u, J\bar{z}_n\right\rangle+\left\|\bar{z}_n\right\|^2\\
&=& \left\|u\right\|^2-2\alpha_n\left\langle u, Jx_n\right\rangle-2(1-\alpha_n)\left\langle u, JS_{j_n}\bar{y}_n\right\rangle\\
&&+\left\|\alpha_n Jx_n+(1-\alpha_n)JS_{j_n}\bar{y}_n\right\|^2\\
&\le& \left\|u\right\|^2-2\alpha_n\left\langle u, Jx_n\right\rangle-2(1-\alpha_n)\left\langle u, JS_{j_n}\bar{y}_n\right\rangle\\
&& +\alpha_n\left\|\bar{y}_n\right\|^2+(1-\alpha_n)\left\|S_{j_n}\bar{y}_n\right\|^2\\
&=&\alpha_n\phi(u,x_n)+(1-\alpha_n)\phi(u,S_{j_n}\bar{y}_n)\\
&\le& \alpha_n\phi(u,x_n)+(1-\alpha_n)\phi(u,\bar{y}_n).
\end{eqnarray*}
Arguing similarly to $(\ref{eq:3})$ and $(\ref{eq:4})$, we obtain
\begin{equation}\label{eq:35}
\phi(u,x_{n+1})\le \phi(u,x_n)-2a(1-\alpha_n) \left(\alpha-\frac{2b}{c^2}\right)||A_{i_n} x_n-A_{i_n}u||^2\le \phi(u,x_n).
\end{equation}
Therefore, the sequence $\left\{\phi(u,x_n)\right\}$ is decreasing. Hence there exists a finite limit of  $\left\{\phi(u,x_n)\right\}$. This together with $(\ref{eq:proPhi})$ and $(\ref{eq:35})$ implies that the sequences $\left\{x_n\right\}$ is bounded and 
\begin{equation}\label{eq:36}
\lim_{n\to\infty}||A_{i_n} x_n-A_{i_n}u||=0.
\end{equation}
Next, we show that $\left\{x_n\right\}$ converges strongly to some point $u$ in $C$. Since the interior of $F$ is nonempty, there exist $p\in F$ and $r>0$ such that
$$
p+rh\in F,
$$
for all $h\in E$ and $\left\|h\right\|\le 1$. Since the sequence $\left\{\phi(u,x_n)\right\}$ is decreasing for all $u\in F$, we have
\begin{equation}\label{eq:38}
\phi(p+rh,x_{n+1})\le \phi(p+rh,x_n).
\end{equation}
From $(\ref{eq:LFP})$, we find that
$$
\phi(u,x_n)=\phi(u,x_{n+1})+\phi(x_{n+1},x_n)+2\left\langle x_{n+1}-u, Jx_n-Jx_{n+1}\right\rangle,
$$
for all $u\in F$. Therefore,
\begin{eqnarray}
\phi(p+rh,x_n)&=&\phi(p+rh,x_{n+1})+\phi(x_{n+1},x_n)\nonumber\\ 
&&+2\left\langle x_{n+1}-(p+rh), Jx_n-Jx_{n+1}\right\rangle.\label{eq:40}
\end{eqnarray}
% \begin{equation}\label{eq:40}
% \phi(p+rh,x_n)=\phi(p+rh,x_{n+1})+\phi(x_{n+1},x_n)+2\left\langle x_{n+1}-(p+rh), Jx_n-Jx_{n+1}\right\rangle.
% \end{equation}
From $(\ref{eq:38}), (\ref{eq:40})$, we obtain
$$
\phi(x_{n+1},x_n)+2\left\langle x_{n+1}-(p+rh), Jx_n-Jx_{n+1}\right\rangle \ge 0.
$$
This inequality is equivalent to
\begin{equation}\label{eq:41*}
\left\langle h, Jx_n-Jx_{n+1}\right\rangle \le \frac{1}{2r}\left\{\phi(x_{n+1},x_n)+2\left\langle x_{n+1}-p, Jx_n-Jx_{n+1}\right\rangle\right\}.
\end{equation}
From $(\ref{eq:LFP})$, we also have 
\begin{equation}\label{eq:42}
\phi(p,x_n)=\phi(p,x_{n+1})+\phi(x_{n+1},x_n)+2\left\langle x_{n+1}-p, Jx_n-Jx_{n+1}\right\rangle.
\end{equation}
From $(\ref{eq:41*}), (\ref{eq:42})$, we obtain
$$
\left\langle h, Jx_n-Jx_{n+1}\right\rangle \le \frac{1}{2r}\left\{\phi(p,x_n)-\phi(p,x_{n+1})\right\},
$$
for all $\left\|h\right\|\le 1$. Hence
$$
\sup_{\left\|h\right\|\le 1}\left\langle h, Jx_n-Jx_{n+1}\right\rangle \le \frac{1}{2r}\left\{\phi(p,x_n)-\phi(p,x_{n+1})\right\}.
$$
The last relation is equivalent to
$$
\left\|Jx_n-Jx_{n+1}\right\|\le \frac{1}{2r}\left\{\phi(p,x_n)-\phi(p,x_{n+1})\right\}.
$$
Therefore, for all $n,m\in N$ and $n>m$, we have
\begin{eqnarray*}
\left\|Jx_n-Jx_{m}\right\|&=&\left\|Jx_n-Jx_{n-1}+Jx_{n-1}-Jx_{n-2}+\ldots+Jx_{m+1}-Jx_{m}\right\|\\ 
& \le& \sum_{i=m}^{n-1} \left\|Jx_{i+1}-Jx_{i}\right\|\\
&\le& \frac{1}{2r}\sum_{i=m}^{n-1}\left\{\phi(p,x_i)-\phi(p,x_{i+1})\right\}\\
&=&\frac{1}{2r}\left(\phi(p,x_m)-\phi(p,x_{n})\right).
\end{eqnarray*}
Letting $m,n\to\infty$, we obtain
$$
\lim_{m,n\to\infty}\left\|Jx_n-Jx_{m}\right\|=0.
$$
Since $E$ is uniformly convex and uniformly smooth Banach space, $J^{-1}$ is uniformly continuous on every bounded subset of $E$.
From the last relation we have
$$
\lim_{m,n\to\infty}\left\|x_n-x_{m}\right\|=0.
$$
Therefore, $\left\{x_n\right\}$ is a Cauchy sequence. Since $E$ is complete and $C$ is closed and convex, $\left\{x_n\right\}$ converges 
strongly to some point $u$ in $C$.
By arguing similarly to $ (\ref{eq:13})$, we obtain 
$$
\phi(x_n,\bar{y}_n)\le\frac{4b^2}{c^2}||A_{i_n} x_n-A_{i_n}u||^2.
$$
This relation together with $ (\ref{eq:36})$ implies that $\phi(x_n,\bar{y}_n)\to 0$. Therefore, $||x_n-\bar{y}_n||\to 0$. By the definition of 
$i_n$, we conclude that $||x_n-y^i_n||\to 0$ for all $1\le i\le M$. Hence,
\begin{equation}\label{eq:45}
\lim_{n\to\infty}y_n^i=u\in C,
\end{equation}
for all $1\le i\le M$.
From $x_{n+1}=\Pi_C \bar{z}_n$ and Lemma $\ref{lem.proProjec}$, we have
\begin{equation}\label{eq:46}
\phi(S_{j_n}\bar{y}_n,x_{n+1})+\phi(x_{n+1},\bar{z}_n)=\phi(S_{j_n}\bar{y}_n,\Pi_C \bar{z}_n)+\phi(\Pi_C \bar{z}_n,\bar{z}_n)\le \phi(S_{j_n}\bar{y}_n,\bar{z}_n).
\end{equation}
Using the convexity of $\left\|.\right\|^2$ we have
\begin{eqnarray*}
\phi(S_{j_n}\bar{y}_n,\bar{z}_n)&=&\left\|S_{j_n}\bar{y}_n\right\|^2-2\left\langle S_{j_n}\bar{y}_n,J\bar{z}_n\right\rangle+\left\|\bar{z}_n\right\|^2\\ 
&= &\left\|S_{j_n}\bar{y}_n\right\|^2-2\alpha_n\left\langle S_{j_n}\bar{y}_n,Jx_n\right\rangle-2(1-\alpha_n)\left\langle S_{j_n}\bar{y}_n,JS_{j_n}\bar{y}_n\right\rangle+\\
&&+\left\|\alpha_n Jx_n+(1-\alpha_n)JS_{j_n} \bar{y}_n\right\|^2\\ 
&\le &\left\|S_{j_n}\bar{y}_n\right\|^2-2\alpha_n\left\langle S_{j_n}\bar{y}_n,Jx_n\right\rangle-2(1-\alpha_n)\left\langle S_{j_n}\bar{y}_n,JS_{j_n}\bar{y}_n\right\rangle+\\
&&+\alpha_n\left\|x_n\right\|^2+(1-\alpha_n)\left\|S_{j_n} \bar{y}_n\right\|^2\\
&=&\alpha_n\phi(S_{j_n}\bar{y}_n,x_n)+(1-\alpha_n)\phi(S_{j_n}\bar{y}_n,S_{j_n} \bar{y}_n)\\
&=&\alpha_n\phi(S_{j_n}\bar{y}_n,x_n).
\end{eqnarray*}
The last inequality together with $(\ref{eq:46})$ implies that 
$$
\phi(x_{n+1},\bar{z}_n)\le \alpha_n\phi(S_{j_n}\bar{y}_n,x_n).
$$
Therefore, from the boundedness of $\left\{\phi(S_{j_n}\bar{y}_n,x_n)\right\}$ and $\lim_{n\to\infty}\alpha_n =0$, we get
$$
\lim_{n\to\infty}\phi(x_{n+1},\bar{z}_n)=0.
$$
Hence, $||x_{n+1}-\bar{z}_n||\to 0$. Since $||x_n-x_{n+1}||\to 0$, we find $||x_n-\bar{z}_n||\to 0$, and by the definition of $j_n$, we 
obtain $||x_n-z_n^j||\to 0$ for all $1\le j\le N$. Thus,
\begin{equation}\label{eq:47}
\lim_{n\to\infty}z_n^j=p.
\end{equation}
By arguing similarly to Steps 4, 5 in the proof of Theorem $\ref{thm:VIP-EP-FPP}$, we obtain 
$$p\in F =\left(\bigcap_{i=1}^M VI(A_i,C)\right)\bigcap \left(\bigcap_{j=1}^N F(S_j)\right).$$
The proof of Theorem $\ref{thm:VIP-EP-FPP3}$ is complete.
\end{proof}
\section{A numerical example}
\setcounter{lemma}{0}
\setcounter{theorem}{0}
\setcounter{equation}{0}
\setcounter{remark}{0}

Consider a Hilbert space $E = \mathbb {R}^1$ with the standart inner product $\left\langle x,y \right\rangle := xy$ and the norm $||x|| := |x| $ for all $x, y \in E$.  Let $C:= [0, 1] \subset E.$ The normalized dual mapping $J=I$ and the Lyapunov functional $\phi (x,y) = |x-y|^2.$ It is well known that, the modulus of convexity of Hilbert space $E$ is $\delta_E(\epsilon) = 1-\sqrt{1-\epsilon^2/4}\ge \frac{1}{4} \epsilon^2$. Therefore, $E$ is $2$-uniformly convex. Moreover, the best constant $\frac{1}{c}$ satisfying relations $|x-y|\le\frac{2}{c^2}|Jx-Jy|=\frac{2}{c^2}|x-y|$ with $0<c\le 1$ is $1$. This implies that $c=1$.
Define the mappings $A_i(x) := x - \frac{x^{i+1}}{i+1},   x \in C,  i =1, \ldots, M,$  and consider the variational inequalities
$$ \left\langle A_i(p^*),p-p^* \right\rangle \ge 0,\quad \forall p\in C,$$
for $i=1, \ldots, M.$
Clearly, $VI(A_i, C) = \{0\},  i =1,\ldots, M.$ Since each mapping $U_i(x) := \frac{x^{i+1}}{i+1}$ is nonexpansive, the mapping $A_i = I - U_i,  i=1,\ldots, M,$ is $\frac{1}{2}$-inverse strongly monotone. Besides, $|A_i(y)| = |A_i(y) - A_i(0)|,$ hence all the assumptions (V1)-(V3) for the variational inequalities are satisfied.\\
Further, let $\{t_i\}_{i=1}^N$ and $\{s_i\}_{i=1}^N$ be two sequences of positive numbers, such that  $0 < t_1 < \ldots < t_N  < 1$ and $  s_i \in (1,  \frac{1}{1- t_i}];    i =1, \ldots, N.$    Define the mappings $S_i: C \to C,    i =1,\ldots, N,$ by putting
$ S_i(x) = 0, $ for $ x \in [0, t_i],$ and $S_i(x) = s_i(x-t_i),$ if $x \in [t_i, 1]$.\\
It is easy to verify that $F(S_i) = \{0\},  \phi(S_i(x), 0)= |S_i(x)|^2 \leq |x|^2 = \phi(x,0) $ for every $x\in C,$ and $|S_i(1) - S_i(t_i)| = s_i 
(1-t_i) > |1-t_i|.$
Hence, the mappings $S_i$ are quasi $\phi$-nonexpansive but not nonexpansive.\\
Finally, let $0<\xi_1 < \ldots < \xi_K< 1$ and $ \eta_k \in (0, \xi_k),   k=1,\ldots, K,$ be two given sequences. Consider $K$ bifunctions $f_k(x,y) := B_k(x)(y-x),  k = 1,\ldots, K,$ where
$ B_k(x) = \frac{\eta_k}{\xi_k}x$ if  $0 \leq x \leq \xi_k,$ and $B_k(x) = \eta_k$ if $\xi_k \leq x \leq 1.$\\
It is easy to verify that all the assumptions (A1)-(A4) for the bifunctions $f_k(x,y)$ are fulfilled. Besides, $EP(f_k) = \{0\}.$  Thus, the solution set
$$F :=\left(\bigcap_{i=1}^M VI(A_i,C)\right)\bigcap \left(\bigcap_{j=1}^N F(S_j)\right)\bigcap\left(\bigcap_{k=1}^K EP(f_k)\right) = \{0\}.$$
According to Theorem $\ref{thm:VIP-EP-FPP4}$, the iteration sequence $\{x_n\}$ generated by
\begin{align*}
&y_n^i=\Pi_C\left(x_n-\lambda_n (x_n -\frac{(x_n)^{i+1}}{i+1})\right), i=1,2,\ldots M, \\
&i_n=\arg\max\left\{|y_n^i-x_n| : i=1,2,\ldots M \right\}, \bar{y}_n=y_n^{i_n},\\
&z_n^j= \alpha_n x_n+(1-\alpha_n)S_j\bar{y}_n, j=1,2,\ldots N,\\
&j_n=\arg\max\left\{|z_n^j-x_n|: j=1,2,\ldots N\right\}, \bar{z}_n=z_n^{j_n},\\
&u_n^k=T_{r_n}^k \bar{z}_n,k=1,2,\ldots K, \\
&k_n=\arg\max\left\{|u_n^k-x_n|: k=1,2,\ldots K\right\}, \bar{u}_n=u_n^{k_n},\\
&C_{n+1}=\left\{z\in C_n:\phi(z,\bar{u}_n)\le \phi(z,\bar{z}_n)\le\phi(z,x_n)\right\},\\
&x_{n+1}=\Pi_{C_{n+1}}x_0,n\ge 0.
\end{align*}
% \begin{equation}\label{example1}
% \left \{
% \begin{array}{ll}
% &y_n^i=\Pi_C\left(x_n-\lambda_n (x_n -\frac{(x_n)^{i+1}}{i+1})\right), i=1,2,\ldots M, \\
% &i_n=\arg\max\left\{|y_n^i-x_n| : i=1,2,\ldots M \right\}, \bar{y}_n=y_n^{i_n},\\
% &z_n^j= \alpha_n x_n+(1-\alpha_n)S_j\bar{y}_n, j=1,2,\ldots N,\\
% &j_n=\arg\max\left\{|z_n^j-x_n|: j=1,2,\ldots N\right\}, \bar{z}_n=z_n^{j_n},\\
% &u_n^k=T_{r_n}^k \bar{z}_n,k=1,2,\ldots K, \\
% &k_n=\arg\max\left\{|u_n^k-x_n|: k=1,2,\ldots K\right\}, \bar{u}_n=u_n^{k_n},\\
% &C_{n+1}=\left\{z\in C_n:\phi(z,\bar{u}_n)\le \phi(z,\bar{z}_n)\le\phi(z,x_n)\right\},\\
% &x_{n+1}=\Pi_{C_{n+1}}x_0,n\ge 0.
% \end{array}
% \right.
% \end{equation}
strongly converges to $x^\dagger := 0.$ \\
A straightforward calculation yields $y_n^i= (1-\lambda_n)x_n-\lambda_n \frac{(x_n)^{i+1}}{i+1}, i=1,2,\ldots M.$\\
Further, the element $u:= u_n^k=T_{r_n}^k \bar{z}_n$ is a solution of the following inequality
\begin{equation}\label{eq:ex3}
(y-u) [B_k(u) +u - \bar{z}_n] \geq 0  \quad \forall y \in [0;1].
\end{equation}
From $(\ref{eq:ex3})$, we find that $u=0$ if and only if  $\bar{z}_n=0$. Therefore, if $\bar{z}_n=0$ then the algorithm stops and $x^\dagger = 0$. If $\bar{z}_n\ne 0$ then inequality $(\ref{eq:ex3})$ is equivalent to a system of inequalities
\begin{equation*}\label{ex4}
\left \{
\begin{array}{ll}
&-u(B_k(u) +u - \bar{z}_n) \geq 0,\\
&(1-u) (B_k(u) +u - \bar{z}_n) \geq 0.
\end{array}
\right.
\end{equation*}
The last system yields $B_k(u) +u =\bar{z}_n$. Hence, if $0<\bar{z}_n\le\eta_k+\xi_k$ then $u_n^k:=u=\frac{\xi_k}{\xi_k+\eta_k}\bar{z}_n$. Otherwise, if $\xi_k+\eta_k<\bar{z}_n\le 1,$ then $u_n^k:=u=\bar{z}_n-\eta_k$. \\
Using the fact that  $F = \{0\} \subset C_{n+1}$ we can conclude that $0 \leq \bar{u}_n \leq \bar{z}_n \leq x_n \leq 1.$ From the definition of $C_{n+1}$ we find
\begin{equation*}\label{eq:ex5}
C_{n+1} = C_n \bigcap [0, \frac{\bar{z}_n + \bar{u}_n}{2}].
\end{equation*}
According to Step 1 of the proof of  Theorem $\ref{thm:VIP-EP-FPP}$, $C_n$ is a closed subset, hence $[0, x_n] \subset C_n$ because $0, x_n \in C_n.$ Further, since $\frac{\bar{u}_n +\bar{z}_n}{2} \leq x_n,$ it implies that $[0, \frac{\bar{u}_n +\bar{z}_n}{2}] \subset [0, x_n] \subset C_n.$ 
Thus, $C_{n+1} = [0, \frac{\bar{u}_n + \bar{z}_n}{2}].$\\
For the sake of comparison between the computing times in the parallel and sequential modes, we choose sufficiently large numbers 
$N, K, M$ and a slowly convergent to zero sequence $\{\alpha_n\}$. \\
\noindent The numerical experiment is performed on a LINUX cluster 1350 with 8 computing nodes. Each node contains two Intel Xeon 
dual core 3.2 GHz, 2GBRam. All the programs are written in C.
For given tolerances we compare the execution time of  algorithm (5.1)  in both parallel 
and sequential modes. We denote by $TOL$- the tolerance $\|x_k - x^*\|$;  $T_p$ - the execution time in parallel mode using 2 CPUs  (in seconds), and $T_s$- the execution time in sequential mode (in seconds). The computing times in both modes are given in Tables 1,  2.\\
 According to Tables 1, 2, in the most favourable cases, the speed-up and the efficiency of the parallel algorithm are $S_p = T_s/T_p \approx 2.0;  E_p =
S_p/2 \approx 1,$ respectively.\\
\begin{table}
\caption{Experiment with $\alpha_n = \frac{1}{\log\left(\log(n+10)\right)}$ }
\label{tab:1}       
\begin{tabular}{lll}
\hline\noalign{\smallskip}
TOL & $T_p$ & $T_s$  \\
\noalign{\smallskip}\hline\noalign{\smallskip}
$10^{-7}$ & 21.84 & 42.42 \\
$10^{-9}$ & 26.46 & 51.24 \\
\noalign{\smallskip}\hline
\end{tabular}
\end{table}
\begin{table}
\caption{Experiment with $\alpha_n = \frac{\log n}{n}$}
\label{tab:2}       % Give a unique label
% For LaTeX tables use
\begin{tabular}{lll}
\hline\noalign{\smallskip}
TOL & $T_p$ & $T_s$  \\
\noalign{\smallskip}\hline\noalign{\smallskip}
 $10^{-7}$ & 6.09 & 11.97 \\
 $10^{-9}$ & 8.10 & 14.28 \\
\noalign{\smallskip}\hline
\end{tabular}
\end{table}
\section{Conclusions} In this paper we proposed two strongly convergent parallel hybrid iterative methods for finding a common element of the set of fixed 
points of quasi $\phi$-asymptotically nonexpansive mappings, the set of solutions of variational  inequalities, and the set of solutions of equilibrium problems  in uniformly smooth and 2-uniformly convex Banach spaces. A numerical example was given to demonstrate the efficiency of the proposed parallel algorithms.
\section{Acknowledgements} 
The research of the first author was partially supported by Vietnam Institute for Advanced Study in Mathematics (VIASM) and  
Vietnam National Foundation for Science and Technology Development (NAFOSTED).

\end{document}